%% file: main.tex
\documentclass[a4paper,10pt,reqno]{amsart}

\UseRawInputEncoding

\usepackage{qworld}

\usepackage{amsmath}
\usepackage{cases}

\usepackage{amsfonts}
\usepackage[colorlinks,linkcolor=blue,citecolor=blue]{hyperref}
\usepackage{latexsym, amssymb, amsmath, amsthm, bbm}
\usepackage[all]{xy}
\usepackage{pgfplots}
\usepackage{enumerate}

\DeclareSymbolFont{EulerExtension}{U}{euex}{m}{n}
\DeclareMathSymbol{\euintop}{\mathop} {EulerExtension}{"52}
\DeclareMathSymbol{\euointop}{\mathop} {EulerExtension}{"48}

\allowdisplaybreaks[4]

\def \id{\operatorname{id}}

\def \C{\mathcal{C}}

\def \e{\varepsilon}

\def \k{\Bbbk}

\def \C{\mathcal{C}}

\def \e{\varepsilon}

\def \C{\mathcal{C}}

\def \ev{\mathsf{ev}}
\def \coev{\mathsf{coev}}

\def \Vec{\mathsf{Vec}}

\def \1{\mathbf{1}}

\usepackage{mathtools}
\usepackage{stmaryrd}

\def \YD{\mathfrak{YD}}

\def\joinrel{\mathrel{\mkern-4mu}}

\newcommand{\blackrtimes}{\mathop{\raisebox{0.2ex}{${\scriptstyle >\joinrel\blacktriangleleft}$}}}

\newcommand{\btd}{\mathop{\raisebox{0.2ex}{${\scriptstyle \blacktriangleright\joinrel\blacktriangleleft}$}}}

\newcommand{\dotrtimes}
{\mathop{\raisebox{0.2ex}{\makebox[0.86em][l]{${\scriptstyle>\joinrel\lessdot}$}}\raisebox{0.12ex}{$\shortmid$}}}

\numberwithin{equation}{section}

\newtheorem{theorem}{Theorem}[section]
\newtheorem{lemma}[theorem]{Lemma}
\newtheorem{proposition}[theorem]{Proposition}
\newtheorem{corollary}[theorem]{Corollary}
\newtheorem{definition}[theorem]{Definition}

\newtheorem{remark}[theorem]{Remark}

\input{figure-layout.tex}
\begin{document}
\title[Characterization results on generalized smash biproduct Hopf algebras]{Characterization results on generalized smash biproduct Hopf algebras over the partial dual construction}

\author[K. Li]{Kangqiao Li}
\address{School of Mathematics, Hangzhou Normal University, Hangzhou 311121, China}
\email{kqli@hznu.edu.cn}


\thanks{$\dag$ This work was supported by National Natural Science Foundation of China [grant number 12301049].}


\begin{abstract}
Let $B$ and $D$ be both algebras and coalgebras in a braided monoidal category $\C$.
In the literature, there are equivalent conditions for $(B,D)$ to form a generalized smash biproduct (or cross product) denoted by $B\times D$, which were mainly given by
Bespalov and Drabant in 1999 as well as by Bulacu, Caenepeel and Torrecillas in 2013.
This paper is devoted to constructing another biproduct $D\times B^\ast$ when $B$ has a left dual object $B^\ast$ in $\C$. As results, we show that $B\times D$ is left smash over $D\times B^\ast$, and estabish a specific isomorhism ${}_{B\times D}\YD(\C)^{B\times D}\approx{}_{D\times B^\ast}\YD(\C)^{D\times B^\ast}$ between the categories of the (left-right) Yetter-Drinfeld modules in $\C$.
The constructions and results are provided in three levels: (1) $B\times D$ is an algebra and a coalgebra; (2) $B\times D$ is a bialgebra; (3) $B\times D$ is a Hopf algebra.
\end{abstract}

\maketitle


\section{Introduction}

As a generalization of semidirect products of groups,
the notion of the smash product algebras $A\rtimes H$ was introduced by Heyneman and Sweedler \cite{HS69} in the late 1960s, where $H$ is a Hopf algebra acting on an algebra $A$.
The dual version was given in \cite{Mol77} by Molnar, and it is called the smash coproduct coalgebra $C\blackrtimes H$ when $C$ is a left $H$-comodule coalgebra.
Later in 1985, Radford \cite{Rad85} provided the condition when the smash product and coproduct are compatible to become a bialgebra (resp. Hopf algebra) $B\dotrtimes H$, which holds if and only if $B$ is a bialgebra (resp. Hopf algebra) in the braided monidal category ${}_H^H\YD$ of (left-left) Yetter-Drinfeld modules over $H$ \cite{Yet90}.
The $B\dotrtimes H$ structure is referred to as the Radford biproduct or bosonization, which plays an important role in classifying Hopf algebras with the dual Chevalley property.

However, if the second tensorand is not required to coincide with the Hopf algebra $H$, there was a generalized structure $B\#D$ introduced by Takeuchi \cite{Tak80}, where $B$ is a left $H$-module algebra and $D$ is a left $H$-comodule algebra. The algebra $B\#D$ is also called the $H$-\textit{smash product algebra} in this paper, and the $H$-\textit{smash coproduct coalgebra} could be defined in a dual form.
More generally, the notion was generalized in \cite{VV94} to a datum $(B,D,\phi)$ satisfying that $B$, $D$ are algebras and $B\otimes D$ becomes an algebra with multiplication
$(m_B\otimes m_D)\circ(\id_B\otimes\phi\otimes\id_D)$,
whose dual version is a datum $(B,D,\psi)$ satisfying that
$B$, $D$ are coalgebras and $B\otimes D$ becomes a coalgebra with comultiplication
$(\id_B\otimes\psi\otimes\id_D)\circ(\Delta_B\otimes \Delta_D)$.

It is known that
 all these notions were defined in an arbitrary braided monoidal category $\C$ (see \cite{Bes95,Bes97,BD98,HS20} etc.).
As for
this paper, we concern the situation when a tuple $(B,D,\phi,\psi)$ in $\C$ makes $B\otimes D$ an algebra as well as a coalgebra, a bialgebra, and even a Hopf algebra.
The object $B\otimes D$ with such structures is referred to as the \textit{generalized smash biproduct} (\cite{CIMZ00,BD01}) or cross product (\cite{BD99,BCT13}), denoted by $B\times D$. In \cite{BD99,CIMZ00,BCT13},
they provided a number of conditions for the tuple at different levels, but here we mention the following one as an example.
Roughly speaking, for \textit{a pair $(B,D)$ of algebras and coalgebras} in a braided monoidal category $\C$, it is shown in \cite{BD99} and \cite{BCT13} that the followings are equivalent on the level of bialgebras:
\begin{itemize}
\item[(1)]
There exist two morphisms
$$\phi:D\otimes B\rightarrow B\otimes D
\;\;\;\;\text{and}\;\;\;\;
\psi:B\otimes D\rightarrow D\otimes B$$
such that $(B,D,\phi,\psi)$ is a bialgebra admissible tuple in the sense that $B\otimes D$ becomes a bialgebra with operations twisted by $\phi$ and $\psi$;

\item[(2)]
There exist a bialgebra $H$ and a diagram
$$\xymatrix{
B \ar@<.5ex>[r]^{\iota} & H \ar@<.5ex>@{-->}[l]^{\zeta} \ar@<.5ex>[r]^{\pi}
& D \ar@<.5ex>@{-->}[l]^{\gamma}  }$$
in $\C$, such that
\begin{itemize}
\item
$\iota$, $\gamma$ are algebra morphisms, and $\pi$, $\zeta$ are coalgebra morphisms;

\item
$\zeta\circ\iota=\id_B$, $\pi\circ\gamma=\id_D$ and
$(\iota\circ\zeta)\ast(\gamma\circ\pi)=\id_H$ hold.
\end{itemize}

\item[(3)]
There are actions and coactions
$$B\in{}_D\C,\;\;\;\;B\in{}^D\C,\;\;\;\;
D\in\C{}_B,\;\;\;\;D\in\C{}^B.$$
from $B$ and $D$ to each others which satisfy a list of compatibility conditions.
\end{itemize}

In this paper, we begin everything with the description (3) above, whose detailed contents are collected in Definition \ref{def:matchedpair}.
To clarify the terms, let us call $(B,D)$ a \textit{Hopf matched pair} in $\C$, if it satisfies the requirements in (3). A weaker notion for $(B,D)$ would be called a \textit{matched pair} when $B\times D$ is an algebra and a coalgebra in $\C$.

Please note that neither $B$ nor $D$ is needed to be bialgebras, and there do exist such examples (e.g. bismash product of matched pair of groups \cite{Tak81,Li23} and the quantum doubles \cite{HKL26} when $\C=\Vec$).
Thus the notion pf the generalized smash biproduct $B\times D$ in $\C$ indeed generalizes the bosonizations, which requires that one tensorand is a Hopf algebra in $\C$ and the other is a Yetter-Drinfeld Hopf algebra.
However, the theory of bosonizations (and Nichols algebras) includes an important concept called the reflections of Yetter-Drienfeld modules. It is based on a structure
referred to as the partially dualized Hopf algebra of bosonizations in $\C$,
which was mainly introduced and studied by Heckenberger and Schneider in \cite{HS13} as well as Barvels, Lentner and Schweigert in \cite{BLS15}.
In this paper, our first main result is the construction denoted by $D\times B^\ast$ as follows, which
should also be a generalization of the partially dualized Hopf algebra of bosonizations.

\begin{proposition}
Let $B\times D$ be a generalized smash biproduct (which is an algebra and a coalgebra) in a braided monoidal category $\C$.
Suppose $B$ has left dual object $B^\ast$.
Then:
\begin{itemize}
\item[(1)]
The object $B^\ast$ has certain structures of an algebra and a coalgebra, such that
$D\times B^\ast$ is also a generalized smash biproduct algebra and coalgebra;
\item[(2)]
$B\times D$ is a bialgebra if and only if $D\times B^\ast$ is a bialgebra in $\C$;
\item[(3)]
If $B\times D$ is a bialgebra and $\id_B$, $\id_D$ are both convolution invertible, then $D\times B^\ast$ is a Hopfalgebra in $\C$.
\end{itemize}
\end{proposition}

This proposition is a collection of Proposition \ref{prop:rightpartialdual1} and Corollary \ref{cor:rightpartialdual2}, and the structure $D\times B^\ast$
would be called the \textit{right partial dual} of the $B\times D$ in $\C$.
In order to prove them, we define and study the categories ${}_D\C^B$ (and ${}^D\C_B$) of \textit{abstract Doi-Hopf modules} over a matched pair $(B,D)$ in Section \ref{section3}. They are analogous to the classical Doi-Hopf modules considered in \cite{Doi92,CMZ97,BT14} etc. However, our notion is defined from a matched pair $(B,D)$ without any Hopf algebras given at first, and the key obsercation is a canonical isomorphism ${}_D\C^B\cong {}_{D\times B^\ast}\C$ of categories.

As the first application (Theorem \ref{thm:cross=smash} and Corollary \ref{cor:cross=smash2}), our construction of the right partial dual $D\times B^\ast$ could help to characterize the original the generalized smash biproduct $B\times D$ as follows:

\begin{theorem}\label{thm1}
Let $(B,D)$ be a matched pair, and let $(B,B^\ast)$ be braided dual pair in a braided monoidal category $\C$. Then:
\begin{itemize}
\item[(1)]
The generalized smash biproduct $B\times D$ is a bialgebra if and only if it is a left smash biproduct bialgebra (over a bialgebra $K$);

\item[(2)]
If $\id_B$ and $\id_D$ are both convolution invertible, then
 $B\times D$ is a left smash biproduct Hopf algebra (over a Hopf algebra $K$).
\end{itemize}
\end{theorem}

We remark that smash biproducts are clearly generalized smash biproducts, and hence Theorem \ref{thm1} states that the converse is also true when $B$ has left dual object in $\C$ (or when the braided monoidal category $\C$ is rigid).

Another consequence (Proposition \ref{prop:YDiso1} and Theorem \ref{thm:YDiso2}) for $D\times B^\ast$ is that we could establish an isomorphism
${}_{B\times D}\YD(\C)^{B\times D}\approx{}_{D\times B^\ast}\YD(\C)^{D\times B^\ast}$ between the categories of left-right Yetter-Drinfeld modules, which is also provided in different levels:

\begin{theorem}\label{thm2}
Let $(B,D)$ be a matched pair, and let $(B,B^\ast)$ is a braided dual pair in a braided monoidal category $\C$. Then there is an isomorphism $\Omega:{}_{B\times D}\YD(\C)^{B\times D}\cong{}_{D\times B^\ast}\YD(\C)^{D\times B^\ast}$ of categories.
Furthermore:
\begin{itemize}
\item[(1)]
If $B\times D$ is a bialgebra, then
$\Omega$ is a lax monoidal functor of prebraided monoidal categories;
\item[(2)]
If $\id_B$ and $\id_D$ are both convolution invertible, then $\Omega$ is an isomorphism of braided monoidal categories.
\end{itemize}
\end{theorem}

Here, the category ${}_{B\times D}\YD(\C)^{B\times D}$ makes sense as long as $(B,D)$ is a matched pair (see Definition \ref{def:YDmod}). Please note that in this case $B\times D$ is an algebra and a coalgebra, but not necessarily a bialgebra.
Moreover, the functor $\Omega$ is constructed in a direct way. Therefore, as a generaliation of the analogous results in \cite{HS13} and \cite{BLS15} for the bosonization  $B\dotrtimes H$ and its partial dualized Hopf algebras, our Theorem \ref{thm2} has a proof without the idea of identification ${}_{B\dotrtimes H}^{B\dotrtimes H}\YD(\C)\approx {}^B_B\YD\big({}^H_H\YD(\C)\big)$.

The paper is mainly presented with graph calculus, and it is organized as follows:
Section \ref{section2} is devoted to recalling and refrasing the known results on the compatibility conditions for (Hopf) matched pairs $(B,D)$, which defines the generalized smash biproduct $B\times D$.
In order to prove our main results, we describe in Section \ref{section3} the categories ${}_D\C^B$ and ${}^D\C_B$ of abstract Doi-Hopf modules .
In Section \ref{section4}, we construct the right partial dual $D\times B^\ast$ and show that $B\times D$ is left smash over $D\times B^\ast$.
Finally in Section \ref{section5},
we describe
the isomorphism between the categories of (left-right) Yetter-Drinfeld modules over $B\times D$ and $D\times B^\ast$.

\section{Preliminaries: Descriptions of (generalized) smash product Hopf algebras}\label{section2}

\subsection{Modules and comodules in a braided monoidal category}

We refer to \cite{Kas95,EGNO15} etc. for basics of braided monoidal categories.

Throughout the paper, let $\C$ be a braided monoidal category with braiding $c$.
We always use notations
$$c_{X,Y}
=\ArxivFigure{0001}
\;\;\;\;\text{and}\;\;\;\;
c_{X,Y}^{-1}
=\ArxivFigure{0002}$$
for all objects $X,Y\in\C$.

For an algebra $B$ in $\C$, we write its multiplication and unit as
$$m_B=\ArxivFigure{0003}
\;\;\;\;\text{and}\;\;\;\;
u_B=\ArxivFigure{0004},$$
respectively.
For a coalgebra $D$ in $\C$, we write its comultiplication and counit as
$$\Delta_D=\ArxivFigure{0005}
\;\;\;\;\text{and}\;\;\;\;
\e_D=\ArxivFigure{0006}.$$

Moreover, we denote:
\begin{itemize}
\item by ${}_B\C$ the category of left $B$-modules in $\C$;
\item by $\C{}_B$ the category of right $B$-modules in $\C$;
\item by ${}^D\C$ the category of left $D$-comodules in $\C$;
\item by $\C{}^D$ the category of right $D$-comodules in $\C$.
\end{itemize}
Also, left $B$-actions, right $B$-actions, left $D$-coactions and right $D$-coactions on $V\in\C$ would be respectively denoted by
$$
\ArxivFigure{0007},\;\;\;\;\;\;\;\;
\ArxivFigure{0008},\;\;\;\;\;\;\;\;
\ArxivFigure{0009},\;\;\;\;\;\;\;\;
\ArxivFigure{0010}.
$$

\subsection{Smash products and smash coproducts in a braided monoidal category}

The following definition could be found in \cite[Section 3.6]{HS20}.

\begin{definition}\label{def:smashprodcoprod}
Let $K$ be a bialgebra in $\C$.
\begin{itemize}
\item[(1)]
Suppose $B$ is an algebra in ${}_K\C$ with $K$-action $p_B$,
and $D$ is an algebra in ${}^K\C$ with $K$-coaction $q_D$. The smash product algebra $B\#D$ is the object $B\otimes D\in\C$ with multiplication $m_{B\#D}$ and unit $u_{B\#D}$ given by
$$
m_{B\#D}
=
\ArxivFigure{0011}
\;\;\;\;\;\;\;\;\text{and}\;\;\;\;\;\;\;\;
u_{B\#D}=\ArxivFigure{0012};
$$

\item[(2)]
Suppose $B$ is a coalgebra in ${}^K\C$ with $K$-coaction $q_B$,
and $D$ is a coalgebra in ${}_K\C$ with $K$-action $p_D$. The smash coproduct coalgebra $B\#D$ is the object $B\otimes D\in\C$ with comultiplication $\Delta_{B\#D}$ and unit $\e_{B\#D}$ given by
$$
\Delta_{B\#D}
=
\ArxivFigure{0013}
\;\;\;\;\;\;\;\;\text{and}\;\;\;\;\;\;\;\;
\e_{B\#D}=\ArxivFigure{0014}.
$$
\end{itemize}
\end{definition}

\begin{definition}\label{def:smashbiprod}
Let $K$ be a bialgebra in $\C$.
Suppose that $B$ is an algebra in ${}_K\C$ as well as a coalgebra in ${}^K\C$, and that $D$ is an algebra in ${}^K\C$ as well as a coalgebra in ${}_K\C$.
With notations in Definition \ref{def:smashprodcoprod}, we say that $B\#D$ is a left $K$-smash biproduct bialgebra (resp. Hopf algebra),
if $(B\#D,m_{B\#D},u_{B\#D},\Delta_{B\#D},\e_{B\#D})$ forms a bialgebra (resp. Hopf algebra) in $\C$.
\end{definition}

%
%
%
%

\subsection{Matched pairs, Hopf matched pairs and generalized smash biproducts}

For convenience, we recall and refrase the notions introduced in \cite[Definition 2.5]{BD99} and \cite[Theorem 5.4]{BCT13} as the following definition.

\begin{definition}(c.f. \cite{BD99} and \cite{BCT13})\label{def:matchedpair}
Let $B$ and $D$ be both algebras and coalgebras in a braided monoidal category $\C$ with structures
\begin{equation}\label{eqn:BDactcoact}
B\in{}_D\C,\;\;\;\;B\in{}^D\C,\;\;\;\;
D\in\C{}_B,\;\;\;\;D\in\C{}^B.
\end{equation}
\begin{itemize}
\item[(1)]
We call $(B,D)$ a matched pair, if the following conditions are staisfied:
\begin{itemize}
\item[(i)]

$A\ArxivFigure{0015}=\id_1$,
$\ArxivFigure{0016}
=\ArxivFigure{0017}$
and
$\ArxivFigure{0018}
=\ArxivFigure{0019}$
hold for all $A\in\{B,D\}$;

\item[(ii)]
$\ArxivFigure{0020}
=\ArxivFigure{0021}
=\ArxivFigure{0022}$
and
$\ArxivFigure{0023}
=\ArxivFigure{0024}
=\ArxivFigure{0025}$;

\item[(iii)]
$\ArxivFigure{0026}
=\ArxivFigure{0027}
=\ArxivFigure{0028}$
and
$\ArxivFigure{0029}
=\ArxivFigure{0030}
=\ArxivFigure{0031}$;

\item[(iv)]
Algebra-coalgebra compatiblity:
$$\ArxivFigure{0032}=
\ArxivFigure{0033},
\;\;\;\;\;\;
\ArxivFigure{0034}=\ArxivFigure{0035};$$

Module-algebra compatibility:
$$\ArxivFigure{0036}
=\ArxivFigure{0037},
\;\;\;\;\;\;
\ArxivFigure{0038}=\ArxivFigure{0039}$$

Comodule-coalgebra compatibility:
$$\ArxivFigure{0040}=\ArxivFigure{0041},
\;\;\;\;\;\;
\ArxivFigure{0042}=\ArxivFigure{0043};
$$

Module-coalgebra compatibility:
$$
\ArxivFigure{0044}=\ArxivFigure{0045},
\;\;\;\;\;\;
\ArxivFigure{0046}=\ArxivFigure{0047}
;$$

Comodule-algebra compatibility:
$$\ArxivFigure{0048}=\ArxivFigure{0049},
\;\;\;\;\;\;
\ArxivFigure{0050}=\ArxivFigure{0051}.
$$
\item[(v)]
Module-comodule compatibility:
\begin{equation}\label{eqn:modcomodaxiom}
\ArxivFigure{0052}
=\ArxivFigure{0053}.
\end{equation}
\end{itemize}

\item[(2)]
A matched pair $(B,D)$ in $\C$ is called a Hopf matched pair, if
one of the following equivalent conditions is satisfied:
$$
\ArxivFigure{0054}
=
\ArxivFigure{0055},\;\;\;\;
\ArxivFigure{0056}
=
\ArxivFigure{0057},
$$

$$
\ArxivFigure{0058}
=
\ArxivFigure{0059},\;\;\;\;
\ArxivFigure{0060}
=
\ArxivFigure{0061}$$
\end{itemize}
\end{definition}

\begin{remark}
A matched pair $(B,D)$ together with the structures (\ref{eqn:BDactcoact}) is called a Hopf datum in \cite{BD99}.
Also, the conditions
(i)-(v) in Definition \ref{def:matchedpair}(1) is
is referred to as the Bespalov-Drabant list by \cite{BCT13}.

Moreover, for a matched pair $(B,D)$, the equivalence of the four equations in Definition \ref{def:matchedpair}(2) is known in \cite[Theorem 5.4]{BCT13}.
\end{remark}

\begin{lemma}(\cite{BD99})\label{lem:crossprod0}
Suppose $(B,D)$ is a matched pair in a braided monoidal category $\C$. Then the tensor product object $B\otimes D$ is an algebra and a coalgebra (denoted by $B\times D$) in $\C$ with the following structures:
\begin{itemize}
\item
multiplication $m_{B\times D}$ and unit $u_{B\times D}$ given by
$$
m_{B\times D}
=\ArxivFigure{0062}
\;\;\;\;\;\;\;\;\text{and}\;\;\;\;\;\;\;\;
u_{B\times D}=\ArxivFigure{0063},
$$

\item
comultiplication $\Delta_{B\times D}$ and unit $\e_{B\times D}$ given by
$$
\Delta_{B\times D}
=\ArxivFigure{0064}
\;\;\;\;\;\;\;\;\text{and}\;\;\;\;\;\;\;\;
\e_{B\times D}=\ArxivFigure{0065}.
$$
\end{itemize}
\end{lemma}

Evidently, we could find:

\begin{corollary}\label{cor:mult&comult}
Suppose $(B,D)$ is a matched pair in braided monoidal category $\C$. Then the following equations hold:
$$
(\e_B\otimes\id_D)\circ m_{B\times D}
=\ArxivFigure{0066},\;\;\;\;
(\id_B\otimes\e_D)\circ m_{B\times D}
=\ArxivFigure{0067}
$$
and
$$
\Delta_{B\times D}\circ(u_B\otimes\id_D)
=\ArxivFigure{0068},\;\;\;\;
\Delta_{B\times D}\circ(\id_B\otimes u_D)
=\ArxivFigure{0069}.
$$
\end{corollary}

\begin{proof}
These equations are directly obtained according to the conditions (i)-(iii) in Definition \ref{def:matchedpair}(1).
\end{proof}

\begin{lemma}(\cite[Theorem 5.4 and Proposition 5.5]{BCT13})\label{lem:crossprod}
Suppose $(B,D)$ is a Hopf matched pair in braided monoidal category $\C$. Then:
\begin{itemize}
\item[(1)]
With structures given in Lemma \ref{lem:crossprod0}, $B\times D$ is a bialgebra in $\C$;

\item[(2)]
If the identity morphisms
$\id_B$ and $\id_D$ are both convolution invertible, then the bialgebra $B\times D$ is a Hopf algebra in $\C$. In this situation, the antipode of $B\times D$ is
$$
S_{B\times D}=\ArxivFigure{0070},$$
where $S_B$ and $S_D$ are the convolution inverses of $\id_B$ and $\id_D$ respectively.
\end{itemize}
\end{lemma}

\begin{remark}
The converse of Lemma \ref{lem:crossprod}(2) is discussed in \cite{BCT13}. Note that it holds when $\C$ is a locally finite $\k$-linear abelian category.
\end{remark}


\begin{lemma}
Let $K$ be a bialgebra in a braided monoidal category $\C$. Suppose $B\#D$ is a (left) $K$-smash biproduct bialgebra.
Then with the notations in Definition \ref{def:smashprodcoprod}, the structures
$$
\ArxivFigure{0071}
=
\ArxivFigure{0072},
\;\;\;\;
\ArxivFigure{0073}
=
\ArxivFigure{0074},
\;\;\;\;
\ArxivFigure{0075}
=
\ArxivFigure{0076},
\;\;\;\;
\ArxivFigure{0077}
=
\ArxivFigure{0078}
$$
make $(B,D)$ a Hopf matched pair in $\C$
\end{lemma}

\begin{lemma}\label{lem:crossprodalg}
Let $(B,D)$ be a matched pair in a braided category $\C$.
Then the following equations hold for $B\times D$:
$$
\ArxivFigure{0079}
=
\ArxivFigure{0080}
,\;\;\;\;
\ArxivFigure{0081}
=
\ArxivFigure{0082}
$$
and
$$
\ArxivFigure{0083}
=
\ArxivFigure{0084}
,\;\;\;\;
\ArxivFigure{0085}
=
\ArxivFigure{0086}
.
$$
In particular,
$$m_{B\times D}\circ (\id_B\otimes u_D\otimes u_B\otimes \id_D)=\id_{B\times D}
=(\id_B\otimes \e_D\otimes \e_B\otimes \id_D)\circ\Delta_{B\times D}.$$
Moreover,
for any $(V,p_V)\in{}_{B\times D}\C$ and $(W,r_W)\in\C^{B\times D}$,
we have
\begin{equation}\label{eqn:B11Dact}
\ArxivFigure{0087}
=
\ArxivFigure{0088},
\;\;\;\;\;\;\;\;
\ArxivFigure{0089}
=
\ArxivFigure{0090},
\end{equation}
and
\begin{equation}\label{eqn:B11Dcoact}
\ArxivFigure{0091}
=
\ArxivFigure{0092},
\;\;\;\;\;\;\;\;
\ArxivFigure{0093}
=
\ArxivFigure{0094}.
\end{equation}
\end{lemma}

\begin{proof}
The desired equations could be shown by the definition of $m_{B\otimes D}$ given in Lemma \ref{lem:crossprod}(1). Specifically,
$$
\ArxivFigure{0095}
=
\ArxivFigure{0096}
=
\ArxivFigure{0097}
=
\ArxivFigure{0098}
=
\ArxivFigure{0099},$$
where the third equality is due to Definition \ref{def:matchedpair}(iii).
The other equation holds by a similar argument.
\end{proof}

\section{Abstract Doi-Hopf modules over Hopf matched pairs}\label{section3}

\begin{definition}
Let $(B,D)$ be a matched pair in a in a braided monoidal category $\C$.
We say that $(V,\mu_V,\nu_V)$ is a left-right abstract Doi-Hopf module over $(B,D)$, if
$(V,\mu_V)\in{}_D\C$ and $(V,\nu_V)\in\C^B$ which satisfy
$$
\ArxivFigure{0100}
=\ArxivFigure{0101}.
$$
Denote by ${}_D\C^B$ the category of all left-right Doi-Hopf modules over $(B,D)$ in $\C$.
\end{definition}

\begin{corollary}
$(B,\mu_B,\Delta_B)$ and $(D,m_D,\nu_D)$ are objects in ${}_D\C^B$.
\end{corollary}

\begin{proposition}\label{prop:DHmodtensorprod}
Suppose $(B,D)$ is a Hopf matched pair in a in a braided monoidal category $\C$.
Then the category ${}_D\C^B$ is a monoidal category with tensor product bifunctor defined as follows:
For any $V,W\in{}_D\C^B$, their tensor product is the object $V\otimes W\in\C$ with
\begin{itemize}
\item
left $D$-action
$\mu_{V\otimes W}
=\ArxivFigure{0102}$
and right $B$-coaction
$\nu_{V\otimes W}
=\ArxivFigure{0103}$.
\end{itemize}
The unit object of ${}_D\C^B$ is $(\1,\mu_\1,\nu_\1)$, where $\mu_\1=\e_D$ and $\nu_\1=u_B$.
\end{proposition}

\begin{proof}
We check this claim into steps:
\begin{itemize}
\item[(1)]
First, let us show the equation
$$\mu_{V\otimes W}\circ(m_D\otimes \id_{V\otimes W})
=\mu_{V\otimes W}\circ(\id_D\otimes \mu_{V\otimes W})$$
by straightforward calculations:
\begin{eqnarray*}
&&
\ArxivFigure{0104}
=
\ArxivFigure{0105}
=
\ArxivFigure{0106}
\\
&&=
\ArxivFigure{0107}
=
\ArxivFigure{0108}
=
\ArxivFigure{0109}
\\
&&=
\ArxivFigure{0110}.
\end{eqnarray*}

\item[(3)]
For any $U,V,W\in{}_D\C^B$, we claim that $(U\otimes V)\otimes W\cong U\otimes (V\otimes W)$ as objects in ${}_D\C^B$. In fact,
\begin{eqnarray*}
&&
\mu_{(U\otimes V)\otimes W}
\\
&=&
\ArxivFigure{0111}
=
\ArxivFigure{0112}
=
\ArxivFigure{0113}
\\
&=&
\ArxivFigure{0114}
=
\ArxivFigure{0115}
=
\ArxivFigure{0116}
\\
&=&
\mu_{U\otimes (V\otimes W)},
\end{eqnarray*}
and one could also prove $\nu_{(U\otimes V)\otimes W}=\nu_{U\otimes (V\otimes W)}$ via the 180-degree rotation of the above graphs.
\end{itemize}
\end{proof}

As one could find that the graphs of the compatibility conditions in Definition \ref{def:matchedpair} are pairwire symmetric. Thus,
we could define the category ${}^D\C_B$, which has completely similar properties.

\begin{definition}
Let $(B,D)$ be a Hopf matched pair in a in a braided monoidal category $\C$.
We say that $(V,\mu'_V,\nu'_V)$ is a right-left abstract Doi-Hopf module over $(B,D)$, if
$(V,\mu'_V)\in\C_B$ and $(V,\nu'_V)\in{}^C\C$ which satisfy
$$
\ArxivFigure{0117}
=\ArxivFigure{0118}.
$$
Denote by ${}^D\C_B$ the category of all right-left Doi-Hopf modules over $(B,D)$ in $\C$.
\end{definition}

\begin{remark}
$(B,m_B,\nu_B)$ and $(D,\mu_D,\Delta_D)$ are objects in ${}^D\C_B$.
\end{remark}

\begin{proposition}\label{prop:DHmodtensorprod2}
The category ${}^D\C_B$ is a monoidal category.
\end{proposition}

\begin{corollary}\label{cor:algcoalginDHmod}
Let $(B,D)$ be a Hopf matched pair in a braided monoidal category $\C$. Then:
\begin{itemize}
\item[(1)]
$(B,m_B,u_B)$ is an algebra in ${}_D\C^B$,
and $(D,m_D,u_D)$ is an algebra in ${}^D\C_B$;

\item[(2)]
$(B,\Delta_B,\e_B)$ is a coalgebra in ${}^D\C_B$,
and $(D,\Delta_D,\e_D)$ is a coalgebra in ${}_D\C^B$;
\end{itemize}
\end{corollary}

\begin{proof}
We prove the first claim in (1) as an example, since others are similar.

In fact, by the definition of $(B\otimes B,\mu_{B\otimes B},\nu_{B\otimes B})$ and $(\1,\mu_1,\nu_\1)$ as objects in ${}_D\C^B$ from Proposition \ref{prop:DHmodtensorprod}, we know that:
\begin{itemize}
\item
The commutativity of the diagrams
\begin{eqnarray*}
\xymatrix{
D\otimes B\otimes B  \ar[r]^{\;\;\;\;\mu_{B\otimes B}}  \ar[d]_{\id_D\otimes m_B}
&  B\otimes B  \ar[d]^{m_B}  \\
D\otimes B  \ar[r]^{\mu_{B}}  &  B
}
&\text{and}&
\xymatrix{
D  \ar[r]^{\mu_{1}}  \ar[d]_{\id_D\otimes u_B}  &  \1  \ar[d]^{u_B}  \\
D\otimes B  \ar[r]^{\mu_{B}}  &  B
}
\end{eqnarray*}
are respectively
equivalent to the first equation in (iv) and the first equation in (iii) of Definition \ref{def:matchedpair}.
Thus $m_B$ and $u_B$ are left $D$-module morphisms;

\item
The commutativity of the diagrams
\begin{eqnarray*}
\xymatrix{
B\otimes B  \ar[r]^{\nu_{B\otimes B}\;\;\;\;}  \ar[d]_{m_B}
&  B\otimes B\otimes B  \ar[d]^{m_B\otimes\id_B}  \\
B  \ar[r]^{\Delta_{B}}  &  B\otimes B
}
&\text{and}&
\xymatrix{
\1  \ar[r]^{\nu_{1}}  \ar[d]_{u_B}  &  B  \ar[d]^{u_B\otimes\id_B}  \\
B  \ar[r]^{\Delta_{B}}  &  B\otimes B
}
\end{eqnarray*}
are respectively
equivalent to the fifth equation in Definition \ref{def:matchedpair}(iv) and
$\Delta_B\circ u_B=u_B\otimes u_B$.
Thus $m_B$ and $u_B$ are right $B$-comodule morphisms.
\end{itemize}
\end{proof}

\section{The right partial dual constructed as a generalized smash biproduct}\label{section4}

\subsection{Braided dual pairs of algebras and coalgebras}

\begin{definition}\label{def:brdualpair}
Let $B$ and $B^\ast$ be both algebras and coalgebras in $\C$. We say that $(B,B^\ast)$
together with two morphisms
$$\omega=\ArxivFigure{0119}
\;\;\;\;\text{and}\;\;\;\;
\omega'=\ArxivFigure{0120}$$
is a braided dual pair in $\C$, if the following conditions hold:

\begin{itemize}
\item[(1)]
Braided Hopf pairing condition:
\begin{equation}\label{eqn:braidedHopfpair}
\ArxivFigure{0121}
=
\ArxivFigure{0122},
\;\;\;\;
\ArxivFigure{0123}
=\ArxivFigure{0124}
,\;\;\;\;
\ArxivFigure{0125}
=
\ArxivFigure{0126}
\;\;\;\;\text{and}\;\;\;\;
\ArxivFigure{0127}
=\ArxivFigure{0128}
;
\end{equation}

\item[(2)]
Non-degeneracy condition:
$$\ArxivFigure{0129}
=\ArxivFigure{0130}
\;\;\;\;\text{and}\;\;\;\;
\ArxivFigure{0131}
=\ArxivFigure{0132}.$$
\end{itemize}
\end{definition}

\begin{corollary}
For a braided dual pair $(B,B^\ast)$ in $\C$, we have
\begin{equation}\label{eqn:braidedHopfcopair}
\ArxivFigure{0133}=
\ArxivFigure{0134},
\;\;\;\;
\ArxivFigure{0135}
=
\ArxivFigure{0136}.
\end{equation}
\end{corollary}

\begin{remark}
For an object $B\in\C$ which is an algebra and a coalgebra, if $B$ has left dual object $B^\ast$, then $(B,B^\ast)$ could become a braided dual pair with $\omega=\ev_B$ and $\omega'_B=\coev_B$.
\end{remark}

\subsection{Construction of the partially dualized Hopf algebra of cross products}

In this subsection, we always assume that $(B,D)$ is a matched pair in a braided monoidal category $\C$, and denote the braiding of $\C$ by $c$.

Suppose $(B,B^\ast)$ is a braided dual pair in the sense of Definition \ref{def:brdualpair} with
$\ArxivFigure{0137}$
and
$\ArxivFigure{0138}$.

\begin{lemma}\label{lem:newmodcomod}
\begin{itemize}
\item[(1)]
$(D,\mu_{D})\in{}_{B^\ast}\C$ and $(D,\nu_{D})\in{}^{B^\ast}\C$ via structures:
$$\mu_{D}=\ArxivFigure{0139}
\;\;\;\;\;\;\;\;\text{and}\;\;\;\;\;\;\;\;
\nu_{D}=\ArxivFigure{0140};
$$

\item[(2)]
$(B^\ast,\mu_{B^\ast})\in\C_D$ and $(B^\ast,\nu_{B^\ast})\in\C^D$ via structures:
$$\mu_{B^\ast}=\ArxivFigure{0141}
\;\;\;\;\;\;\;\;\text{and}\;\;\;\;\;\;\;\;
\nu_{B^\ast}
=\ArxivFigure{0142}.
$$
\end{itemize}
\end{lemma}

\begin{proof}
As $(B,D)$ is a matched pair in $\C$, it follows by Definition \ref{def:matchedpair}(1) that we are given structures
$$B\in{}_D\C,\;\;\;\;B\in{}^D\C,\;\;\;\;
D\in\C{}_B,\;\;\;\;D\in\C{}^B.$$
Let us verify that $\mu_D$ is a left $B^\ast$-module structure by direct calculations, while other structures could be checked similarly :
$$
\ArxivFigure{0143}
=
\ArxivFigure{0144}
=
\ArxivFigure{0145}
=
\ArxivFigure{0146}
=
\ArxivFigure{0147},
$$
where the third equality is due to $D\in\C^B$.
Moreover,
$$
\ArxivFigure{0148}
=
\ArxivFigure{0149}
=
\ArxivFigure{0150}
=
\ArxivFigure{0151}.
$$


%
%
\end{proof}

The goal of this subsection is to prove that $(D,B^\ast)$ ia a matched pair
with the structures $\mu_D$, $\nu_D$, $\mu_{B^\ast}$ and $\nu_{B^\ast}$ defined in Lemma \ref{lem:newmodcomod}, and
then a generalized smash biproduct $D\times B^\ast$ is formulated to be an algebra and a coalgebra.

\begin{lemma}\label{lem:munumixed}
Let $(B,D)$ be a matched pair, and let $(B,B^\ast)$ be a braided dual pair. Then the following equations hold:
\begin{equation}\label{eqn:mumu&nunu}
\ArxivFigure{0152}
=
\ArxivFigure{0153},
\;\;\;\;\;\;\;\;
\ArxivFigure{0154}
=
\ArxivFigure{0155},
\end{equation}
and
\begin{equation}\label{eqn:munu&numu}
\ArxivFigure{0156}
=
\ArxivFigure{0157},
\;\;\;\;\;\;\;\;
\ArxivFigure{0158}
=
\ArxivFigure{0159}.
\end{equation}
\end{lemma}

\begin{proof}
These could all be shown by using Equations (\ref{eqn:braidedHopfpair}) and (\ref{eqn:braidedHopfcopair}). Specifically,
$$
\ArxivFigure{0160}
=
\ArxivFigure{0161}
=
\ArxivFigure{0162}
=
\ArxivFigure{0163},
$$

$$
\ArxivFigure{0164}
=
\ArxivFigure{0165}
=
\ArxivFigure{0166}
=
\ArxivFigure{0167}.
$$

$$
\ArxivFigure{0168}
=
\ArxivFigure{0169}
=
\ArxivFigure{0170}
=
\ArxivFigure{0171},
$$

$$
\ArxivFigure{0172}
=
\ArxivFigure{0173}
=
\ArxivFigure{0174}
=
\ArxivFigure{0175}.
$$
\end{proof}

\begin{proposition}\label{prop:rightpartialdual1}
Let $(B,D)$ be a matched pair, and let $(B,B^\ast)$ be a braided dual pair. Then
$(D,B^\ast)$ is a matched pair in $\C$.
\end{proposition}

\begin{proof}
By the definition of matched pairs,
we aim to check the conditions (i)-(v) in Definition \ref{def:matchedpair}(1) for $(D,B^\ast)$ in steps.
The formulas in Lemma \ref{lem:munumixed} would be used frequently.

Note that the axioms (i) to (iii) are evident.

\begin{itemize}
\item[(iv)]
The former equation of the algebra-coalgebra compatibility for $(D,B^\ast)$ is shown as follows:
\begin{eqnarray*}
&&
\ArxivFigure{0176}
=
\ArxivFigure{0177}
=
\ArxivFigure{0178}
=
\ArxivFigure{0179}
\\
&&=
\ArxivFigure{0180}
=
\ArxivFigure{0181}
=
\ArxivFigure{0182}.
\end{eqnarray*}

The latter equation of the algebra-coalgebra compatibility for $(D,B^\ast)$ is shown as follows:

\begin{eqnarray*}
&& \ArxivFigure{0183}
=
\ArxivFigure{0184}
=
\ArxivFigure{0185}
=
\ArxivFigure{0186}
\\
&&=
\ArxivFigure{0187}
=
\ArxivFigure{0188}
=
\ArxivFigure{0189}
\\
&&=
\ArxivFigure{0190}
=
\ArxivFigure{0191}
=
\ArxivFigure{0192}.
\end{eqnarray*}

The former equation of the module-algebra compatibility for $(D,B^\ast)$ is shown as follows:
\begin{eqnarray*}
&& \ArxivFigure{0193}
=
\ArxivFigure{0194}
=
\ArxivFigure{0195}
=
\ArxivFigure{0196}
=
\ArxivFigure{0197}  \\
&=&
\ArxivFigure{0198}
=
\ArxivFigure{0199}
=
\ArxivFigure{0200}
=
\ArxivFigure{0201}
\end{eqnarray*}

The latter equation of the module-algebra compatibility for $(D,B^\ast)$ is shown as follows:
\begin{eqnarray*}
&&
\ArxivFigure{0202}
=
\ArxivFigure{0203}
=
\ArxivFigure{0204}
=
\ArxivFigure{0205}  \\
&&
=
\ArxivFigure{0206}
=
\ArxivFigure{0207}
=
\ArxivFigure{0208}  \\
&&
=
\ArxivFigure{0209}
=
\ArxivFigure{0210}
=
\ArxivFigure{0211}
=
\ArxivFigure{0212}.
\end{eqnarray*}

The former equation of the comodule-coalgebra compatibility for $(D,B^\ast)$ is shown as follows:
\begin{eqnarray*}
&&
\ArxivFigure{0213}
=
\ArxivFigure{0214}
=
\ArxivFigure{0215}
=
\ArxivFigure{0216}
=
\ArxivFigure{0217}
\\
&&=
\ArxivFigure{0218}
=
\ArxivFigure{0219}
=
\ArxivFigure{0220}
=
\ArxivFigure{0221}
\\
&&=
\ArxivFigure{0222}
=
\ArxivFigure{0223}
=
\ArxivFigure{0224}.
\end{eqnarray*}

The latter equation of the comodule-coalgebra compatibility for $(D,B^\ast)$ is shown as follows:
\begin{eqnarray*}
&&
\ArxivFigure{0225}
=
\ArxivFigure{0226}
=
\ArxivFigure{0227}
\overset{(\ref{eqn:braidedHopfcopair})}=
\ArxivFigure{0228}
\\
&&=
\ArxivFigure{0229}
=
\ArxivFigure{0230}
=
\ArxivFigure{0231}
\\
&&=
\ArxivFigure{0232}
=
\ArxivFigure{0233}
=
\ArxivFigure{0234}
\\
&&=
\ArxivFigure{0235}
=
\ArxivFigure{0236}
=
\ArxivFigure{0237}.
\end{eqnarray*}

The former equation of the module-coalgebra compatibility for $(D,B^\ast)$ is shown as follows:
\begin{eqnarray*}
&&
\ArxivFigure{0238}
=
\ArxivFigure{0239}
=
\ArxivFigure{0240}
=
\ArxivFigure{0241}
\\
&&=
\ArxivFigure{0242}
=
\ArxivFigure{0243}
=
\ArxivFigure{0244},
\end{eqnarray*}
where the fourth equality is due to the fourth equation of (iv) in Definition \ref{def:matchedpair}(1).

The latter equation of the module-coalgebra compatibility for $(D,B^\ast)$ is shown as follows:
\begin{eqnarray*}
&&
\ArxivFigure{0245}
=
\ArxivFigure{0246}
=
\ArxivFigure{0247}
=
\ArxivFigure{0248}
\\
&&=
\ArxivFigure{0249}
=
\ArxivFigure{0250}
=
\ArxivFigure{0251}
=
\ArxivFigure{0252},
\end{eqnarray*}
where the fifth equality is due to the first equation of (iv) in Definition \ref{def:matchedpair}(1).

The former equation of the module-coalgebra compatibility for $(D,B^\ast)$ is shown as follows:
\begin{eqnarray*}
&&
\ArxivFigure{0253}
=
\ArxivFigure{0254}
=
\ArxivFigure{0255}
=
\ArxivFigure{0256}
\\
&&=
\ArxivFigure{0257}
=
\ArxivFigure{0258},
\end{eqnarray*}
where the penultimate equality is due to the second equation of (iv) in Definition \ref{def:matchedpair}(1).

The latter equation of the comodule-algebra compatibility for $(D,B^\ast)$ is shown as follows:
\begin{eqnarray*}
&&
\ArxivFigure{0259}
=
\ArxivFigure{0260}
=
\ArxivFigure{0261}
=
\ArxivFigure{0262}
\\
&&=
\ArxivFigure{0263}
=
\ArxivFigure{0264}
=
\ArxivFigure{0265}
=
\ArxivFigure{0266},
\end{eqnarray*}
where the fifth equality is due to the third equation of (iv) in Definition \ref{def:matchedpair}(1).

\item[(v)]
Using Lemma \ref{lem:munumixed}(1), we check the module-comodule compatibility for $(D,B^\ast)$ by the following calculations:
\begin{eqnarray*}
&&
\ArxivFigure{0267}
=
\ArxivFigure{0268}
=
\ArxivFigure{0269}
=
\ArxivFigure{0270}
\\
&&=
\ArxivFigure{0271}
=
\ArxivFigure{0272}
=
\ArxivFigure{0273},
\end{eqnarray*}
where the penultimate equality could be checked directly.
\end{itemize}

As a conclusion, the pair $(D,B^\ast)$ satisfies the conditions of matched pair. It follows that the generalized smash biproduct $D\times B^\ast$ is an algebra as well as a coalgebra.
\end{proof}

\subsection{Abstract Doi-Hopf modules over a Hopf matched pair}

In this subsection, we would use
the following immediate lemma. Its inverse claim is the reconstruction of Hopf algebras in a braided monoidal category, which is introduced in \cite{Maj93}

\begin{lemma}\label{lem:fibertoC}
Let $K$ be an algebra as well as a coalgebra in a braided monoidal category $\C$. Suppose
${}_K\C$ is a monoidal category with tensor product bifunctor defined as follows:
For any $V,W\in{}_K\C$, their tensor product is the object $V\otimes W\in\C$ with left $K$-action
$$p_{V\otimes W}
:=
\ArxivFigure{0274}.$$
If the forgetful functor $\mathbf{U}:{}_K\C\rightarrow\C$ is (strictly) monoidal, then $K$ is a bialgebra in $\C$.
\end{lemma}

\begin{proposition}\label{prop:monoidaliso1}
Let $(B,D)$ be a matched pair, and let $(B,B^\ast)$ be a braided dual pair in a braided monoidal category $\C$.
Then:
\begin{itemize}
\item[(1)]
There is an isomorphism $\Phi:{}_D\C^B\cong{}_{D\times B^\ast}\C$ of categories such that
$\Phi(V):=V$ is a left $D\times B^\ast$-module with structure
$$p_V=\ArxivFigure{0275}$$
for each $V\in{}_{D\times B^\ast}\C$.

\item[(2)]
If $(B,D)$ is a Hopf matched pair, then
${}_{D\times B^\ast}\C$ is a monoidal category, and
$\Phi$ is an isomorphism of monoidal categories satisfying $\mathbf{U}\circ\Phi^{-1}$ is (strictly) monoidal.
\end{itemize}
\end{proposition}

\begin{proof}
(2) is a direct consequence of (1).

At first we check that $p_V$ a left $D\times B^\ast$-module structure for
$V\in{}_{D\times B^\ast}\C$.
 In fact, one could use Lemma \ref{lem:munumixed}(1) to calculate:
\begin{eqnarray*}
&&
\ArxivFigure{0276}
=
\ArxivFigure{0277}
=
\ArxivFigure{0278}
\\
&&=
\ArxivFigure{0279}
=
\ArxivFigure{0280}
\\
&&=
\ArxivFigure{0281}
=
\ArxivFigure{0282}
=
\ArxivFigure{0283},
\end{eqnarray*}
where the penultimate equality is due to the compatibility relation of $V$ as an abstract Doi-Hopf module.

Conversely,
for each $(V,p_V)\in{}_{D\times B^\ast}\C$, define $\overline{\Phi}(V,p_V)=(V,\mu_V,\nu_V)$, where:
$$
\mu_V=\ArxivFigure{0284}
\;\;\;\;\text{and}\;\;\;\;
\nu_V=\ArxivFigure{0285}.$$

Then by applying Lemma \ref{lem:crossprodalg}(1) to $D\times B^\ast$, one could find that
$(V,\mu_V)\in{}_D\C$. Similarly, $V$ has a left $B^\ast$-action
$\ArxivFigure{0286}$, and hence $\nu_V$ is a right $B$-comodule structure.

Furthermore, $(V,\mu_V,\nu_V)$ is an abstract Doi-Hopf module in $\C$. In fact, we have following calculations:
\begin{eqnarray*}
&&
\ArxivFigure{0287}
=
\ArxivFigure{0288}
=
\ArxivFigure{0289}
=
\ArxivFigure{0290}
\\
&&=
\ArxivFigure{0291}
=
\ArxivFigure{0292}
=
\ArxivFigure{0293}
\\
&&
\overset{(\ref{eqn:B11Dact})}{=}
\ArxivFigure{0294}
=
\ArxivFigure{0295}
=
\ArxivFigure{0296}.
\end{eqnarray*}

Moreover, it is straightforward to find that the functors $\Phi$ and $\overline{\Phi}$ are mutually inverse.

Finally, we aim to show that $\Phi(V)\otimes\Phi(W)=\Phi(V\otimes W)$ as $D\times B^\ast$-modules in $\C$ for any $V,W\in{}_D\C^B$, namely,
\begin{equation}\label{eqn:Phimonoidal}
\ArxivFigure{0297}
=
\ArxivFigure{0298}.
\end{equation}
However, ccording to Lemma \ref{lem:crossprodalg}(2), it suffices to check that $\Phi(V)\otimes\Phi(W)$ and $\Phi(V\otimes W)$ have the same $D\times 1$-actions and $1\times B$-actions.
In fact, we have calculations:
\begin{eqnarray*}
&&
\ArxivFigure{0299}
=
\ArxivFigure{0300}
=
\ArxivFigure{0301}
\\
&&=
\ArxivFigure{0302}
=
\ArxivFigure{0303}
=
\ArxivFigure{0304},
\end{eqnarray*}
as well as
\begin{eqnarray*}
&&
\ArxivFigure{0305}
=
\ArxivFigure{0306}
=
\ArxivFigure{0307}
\\
&&=
\ArxivFigure{0308}
=
\ArxivFigure{0309}
=
\ArxivFigure{0310},
=
\ArxivFigure{0311}.
\end{eqnarray*}
\end{proof}

\begin{corollary}\label{cor:rightpartialdual2}
Let $(B,D)$ be a Hopf matched pair, and let $(B,B^\ast)$ be a braided dual pair in a braided monoidal category $\C$. Then:
\begin{itemize}
\item[(1)]
The generalized smash biproduct $D\times B^\ast$ is a bialgebra in $\C$;
\item[(2)]
If $\id_B$ and $\id_D$ are both convolution invertible, then $D\times B^\ast$ is Hopf algebra in $\C$.
\end{itemize}
\end{corollary}

\begin{proof}
The first claim is due to Lemma \ref{lem:crossprod}(1).

Moreover, the second claim is a consequence of Lemma \ref{lem:crossprod}(2). Specifically, let $S_B$ be the convolution inverse of $\id_B$. Then it is straightforward to check that
$\ArxivFigure{0312}$
is the convolution inverse of $\id_{B^\ast}$, and hence $D\times B^\ast$ is also a Hopf algebra.
\end{proof}

\begin{definition}
Let $(B,D)$ be a matched pair, and let $(B,B^\ast)$ be a braided dual pair in a braided monoidal category $\C$.
The generalized smash biproduct $D\times B^\ast$ is called the right partial dual of  $B\times D$.
\end{definition}

\begin{remark}
If $\C$ is the category of finite-dimensional vector spaces over a filed $\k$, then $D\times B^\ast$ is exactly the right partially dualized Hopf algebra $D\btd B^\ast$ of the Hopf algebra $B\times D$ over $\k$ introduced in \cite{Li23}.
\end{remark}

\subsection{Smash biproduct description}

\begin{lemma}\label{lem:monoidaliso2}
Let $(B,D)$ be a Hopf matched pair, and let $(B,B^\ast)$ be a braided dual pair in a braided monoidal category $\C$.
Then there exists an isomorphism
$\Psi:{}^D\C_B\cong{}^{D\times B^\ast}\C$ of monoidal categories, such that $\mathbf{U}\circ\Psi^{-1}$ is (strictly) monoidal.
\end{lemma}

\begin{proof}
This is analogous to the proof of Proposition \ref{prop:monoidaliso1}, while the mutually inverse functors are given as follows.

For each $W\in{}^D\C_B$, define $\Psi(W)=W$ with left $D\times B^\ast$-coaction
$$\ArxivFigure{0313}.$$

Conversely,
for each $(W,q_W)\in{}^{D\times B^\ast}\C$, denote
$$
\nu'_{W}=
\ArxivFigure{0314}
\;\;\;\;\text{and}\;\;\;\;
\mu'_{W}=
\ArxivFigure{0315}.
$$
Then one could directly verify that $(W,\nu'_W,\mu'_W)\in{}^D\C_B$.
\end{proof}

\begin{proposition}\label{prop:Kmodcomodstru}
Let $(B,D)$ be a Hopf matched pair in a braided monoidal category $\C$.
Suppose $(B,B^\ast)$ is a braided dual pair, and denote by $K:=D\times B^\ast$ the generalized smash biproduct Hopf algebra.
Then:
\begin{itemize}
\item[(1)]
$B$ is a left $K$-module algebra with action
$p_B=\ArxivFigure{0316}$,
and $D$ is a left $K$-comodule algebra with coaction
$q_D=\ArxivFigure{0317}$;

\item[(2)]
$B$ is a left $K$-comodule coalgebra with coaction
$q_B=\ArxivFigure{0318}$,
and $D$ is a left $K$-module coalgebra with action
$p_D=\ArxivFigure{0319}$.
\end{itemize}
\end{proposition}

\begin{proof}
\begin{itemize}
\item[(1)]
Recall that $(B,\mu_B,\Delta_B)\in{}_D\C^B$, whose image under $\Phi$ is
$(B,p_B)$. It follows that $(B,p_B)\in{}_{D\times B^\ast}\C$
according to Proposition \ref{prop:monoidaliso1}(2) that
$\Phi:{}_D\C^B\cong{}_{D\times B^\ast}\C$ is an isomorphism of monoidal categories,

\item[(2)]
Similar by Lemma \ref{lem:monoidaliso2}.

\end{itemize}
\end{proof}

\begin{theorem}\label{thm:cross=smash}
Let $(B,D)$ be a Hopf matched pair, and let $(B,B^\ast)$ be braided dual pair in a braided monoidal category $\C$.
Then the structures in Proposition \ref{prop:Kmodcomodstru} make $B\#D$ a left $K$-smash biproduct, which coinsides with the generalized smash biproduct $B\times D$.
\end{theorem}

\begin{proof}
Firstly, we know by Corollary \ref{cor:algcoalginDHmod} and the monoidal isomorphisms $\Phi$ and $\Psi$ that:
\begin{itemize}
\item
$B$ is an algebra in ${}_K\C$ with $K$-action $p_B$,
and $D$ is an algebra in ${}^K\C$ with $K$-coaction $q_D$;

\item
$B$ is a coalgebra in ${}^K\C$ with $K$-coaction $q_B$,
and $D$ is a coalgebra in ${}_K\C$ with $K$-action $p_D$.
\end{itemize}

Then
our goal is to show that all the structures of $B\#D$ are the same with $B\times D$ as algebras and coalgebras.

Specifically, according to Definitions \ref{def:smashprodcoprod} and \ref{def:smashbiprod}, we know that the multiplication of $B\#D$ is
\begin{eqnarray*}
&&
m_{B\#D}
=
\ArxivFigure{0320}
=
\ArxivFigure{0321}
=
\ArxivFigure{0322}
\\
&&=
\ArxivFigure{0323}
=
\ArxivFigure{0324}
=
m_{B\times D},
\end{eqnarray*}
where the last equality is due to Lemma \ref{lem:crossprod}(1).

On the other hand, the comultilication of $B\#D$ is
\begin{eqnarray*}
&&
\Delta_{B\#D}
=
\ArxivFigure{0325}
=
\ArxivFigure{0326}
=
\ArxivFigure{0327}
\\
&&=
\ArxivFigure{0328}
=
\ArxivFigure{0329}
=\Delta_{B\times D},
\end{eqnarray*}
where the last equality is also due to Lemma \ref{lem:crossprod}(1).

As for unit and counit, it is straightforward to find that $u_{B\#D}=u_{B\times D}$ and $\e_{B\#D}=\e_{B\times D}$ by definitions.
\end{proof}

\begin{corollary}\label{cor:cross=smash2}
Under the assumptions in Theorem \ref{thm:cross=smash}, we have:
\begin{itemize}
\item[(1)]
The generalized smash biproduct $B\times D$ is a bialgebra if and only if it is a smash biproduct bialgebra;

\item[(2)]
If $\id_B$ and $\id_D$ are both convolution invertible, then
 $B\times D$ is a smash biproduct Hopf algebra over a Hopf algebra.
\end{itemize}
\end{corollary}

\section{Yetter-Drinfeld modules over generalized smash biproducts}\label{section5}

In this section, we show that any generalized smash biproduct $B\times D$ and its right partial dual $D\times B^\ast$ the have equivalent categories of (left-right) Yetter-Drinfeld modules. This could be regarded as a generalization of the results in \cite{HS13} and \cite{BLS15}.

\subsection{Yetter-Drinfeld modules over (generalized) smash biproduct algebras and coalgebras}

\begin{definition}(c.f. \cite[Figure 5a]{Bes97})\label{def:YDmod}
Let $\C$ be a braided monoidal category, and let $H$ be an algebra and a coalgebra in $\C$.
A (left-right) Yetter-Drinfeld module over $H$ in $\C$ is a triple $(M,p_M,r_M)$ satisfying $(M,p_M)\in{}_{H}\C$, $(M,r_M)\in\C^H$ and
\begin{equation}\label{eqn:YDcompatibility}
\ArxivFigure{0330}
=
\ArxivFigure{0331},
\end{equation}
where $p_M=\ArxivFigure{0332}$
 and $r_M=\ArxivFigure{0333}$.
Denote by ${}_H\YD(\C)^H$ the category of all (left-right) Yetter-Drinfeld module over $H$ in $\C$.
\end{definition}

%

Now we consider the case when $H=B\times D$ is the generalized smash biproduct.
One could infer from the compatibility condition (\ref{eqn:YDcompatibility}) to obtain the following formulas:

\begin{lemma}\label{lem:YDrestricted}
Let $(B,D)$ be a matched pair in a braided monoidal category. Suppose that
$(M,p_M)\in{}_{B\times D}\C$ and $(M,r_M)\in\C^{B\times D}$. Then
$(M,p_M,r_M)\in{}_{B\times D}\YD(\C)^{B\times D}$ if and only if the following four equations hold:
\begin{equation}\label{eqn:YDrestricted}
\ArxivFigure{0334}
=
\ArxivFigure{0335},
\;\;\;\;\;\;\;\;
\ArxivFigure{0336}
=
\ArxivFigure{0337},
\end{equation}
as well as
\begin{equation}\label{eqn:YDrestricted2}
\ArxivFigure{0338}
=
\ArxivFigure{0339},
\;\;\;\;\;\;\;\;
\ArxivFigure{0340}
=
\ArxivFigure{0341}.
\end{equation}
\end{lemma}

\begin{proof}
The ``only if'' part is evident. In order to prove the ``if part'', we check that the condition (\ref{eqn:YDcompatibility}) holds for $H=B\times D$ according to the following calculations:
\begin{eqnarray*}
&&
\ArxivFigure{0342}
\overset{(\ref{eqn:B11Dact}),\;(\ref{eqn:B11Dcoact})}=
\ArxivFigure{0343}
\\
&&=
\ArxivFigure{0344}
\overset{(\ref{eqn:YDrestricted})}=
\ArxivFigure{0345}
\\
&&\overset{(\ref{eqn:YDrestricted2})}=
\ArxivFigure{0346}
\overset{(\ref{eqn:YDrestricted})}=
\ArxivFigure{0347}
\\
&&=
\ArxivFigure{0348}
\overset{(\ref{eqn:B11Dact}),\;(\ref{eqn:B11Dcoact})}=
\ArxivFigure{0349}.
\end{eqnarray*}

\end{proof}

\begin{lemma}\label{lem:OmegaMmodcomod}
For each $(M,p_M,r_M)\in{}_{B\times D}\YD(\C)^{B\times D}$, denote
$$p'_M=
\ArxivFigure{0350}
\;\;\;\;\text{and}\;\;\;\;
r'_M=
\ArxivFigure{0351}.$$
Then we have $(M,p'_M)\in{}_{D\times B^\ast}\C$ and $(M,r'_M)\in\C^{D\times B^\ast}$.
\end{lemma}

\begin{proof}
Firstly, it follows by the former equation of (\ref{eqn:YDrestricted}) that the object $(M,\mu_M,\nu_M)$ belongs to ${}_D\C^B$, where
$\mu_M=
\ArxivFigure{0352}$
and
$\nu_M=
\ArxivFigure{0353}$.
However, we could find that the isomorphism $\Phi:{}_D\C^B\cong{}_{D\times B^\ast}\C$ introduced in Proposition \ref{prop:monoidaliso1}(1) sends $(M,\mu_M,\nu_M)$ to the object $(M,p'_M)$, which implies that $p'_M$ is a left $D\times B^\ast$-module structure.

On the other hand, by direct calculations according to the latter equation of (\ref{eqn:YDrestricted}), one could know that $(M,r'_M)\in\C^{D\times B^\ast}$. Specifically,
\begin{eqnarray*}
&&
\ArxivFigure{0354}
=
\ArxivFigure{0355}
=
\ArxivFigure{0356}
\\
&&
\overset{(\ref{eqn:YDcompatibility})}=
\ArxivFigure{0357}
=
\ArxivFigure{0358}
\\
&&=
\ArxivFigure{0359}
\overset{(\ref{eqn:mumu&nunu})}=
\ArxivFigure{0360}
=(\id_M\otimes\Delta_{D\times B^\ast})\circ r_M.
\end{eqnarray*}

\end{proof}

\begin{proposition}\label{prop:YDiso1}
There is an isomorphism $\Omega:{}_{B\times D}\YD(\C)^{B\times D}\cong{}_{D\times B^\ast}\YD(\C)^{D\times B^\ast}$ of categories. Specifically, with the notations in Lemma \ref{lem:OmegaMmodcomod},
$$\Omega(M,p_M,r_M):=(M,p'_M,r'_M)\;\;\;\;\;\;\;\;
(\forall M\in{}_{B\times D}\YD(\C)^{B\times D}),$$
and $\Omega$ is the identity on morphisms.
\end{proposition}

\begin{proof}
Since $(D,B^\ast)$ is a matched pair,
we could use Lemma \ref{lem:YDrestricted} to prove that
$\Omega(M,p_M,r_M):=(M,p'_M,r'_M)$ is an object ${}_{D\times B^\ast}\YD(\C)^{D\times B^\ast}$.

In fact,
the former equation in (\ref{eqn:YDrestricted}) for $(D,B^\ast)$ is shown as follows:
\begin{eqnarray*}
&&
\ArxivFigure{0361}
=
\ArxivFigure{0362}
\overset{(\ref{eqn:munu&numu})}=
\ArxivFigure{0363}
\\
&&=
\ArxivFigure{0364}
=
\ArxivFigure{0365}
=
\ArxivFigure{0366}
\\
&&=
\ArxivFigure{0367}
=
\ArxivFigure{0368}
=
\ArxivFigure{0369},
\end{eqnarray*}
while other three equations could be shown similarly.

On the other hand, the inverse functor of $\Omega$ is found as follows: For each $(M,p'_M,r'_M)\in{}_{D\times B^\ast}\YD(\C)^{D\times B^\ast}$, we define
$\Omega^{-1}(M,p'_M,r'_M)=(M,p_M,r_M)$, where
$$p_M=
\ArxivFigure{0370}
\;\;\;\;\text{and}\;\;\;\;
r_M=
\ArxivFigure{0371}.$$
\end{proof}

\subsection{Yetter-Drinfeld modules over (generalized) smash biproduct bialgebras}

\begin{lemma}(\cite[Theorem 3.4.3]{Bes97} and \cite[Remark 3.6(1)]{BLS15})\label{lem:YDmodmonoidal}
Let $H$ be a bialgebra in a braided monoidal category $\C$. Then:
\begin{itemize}
\item[(1)]
The category ${}_H\YD(\C)^H$ is a monoidal category with tensor product bifunctor defined as follows:
For any $M,N\in{}_H\YD(\C)^H$, their tensor product is the object $M\otimes N\in\C$ with
\begin{itemize}
\item
left $H$-action
$p_{M\otimes N}=
\ArxivFigure{0372}$
and right $H$-coaction
$r_{M\otimes N}=
\ArxivFigure{0373}$.
\end{itemize}
The unit object of ${}_H\YD(\C)^H$ is $(\1,\e_H,u_H)$.

Furthermore,
${}_H\YD(\C)^H$ is prebraided with $\sigma$ defined via
$$\sigma_{M,N}:=
\ArxivFigure{0374},
\;\;\;\;\;\;\;\;
(\forall M,N\in{}_H\YD(\C)^H).$$

\item[(2)]
If $H$ is a Hopf algebra, then ${}_H\YD(\C)^H$ is a braided monoidal category with braiding $\sigma$.

\end{itemize}
\end{lemma}

\begin{theorem}\label{thm:YDiso2}
Let $(B,D)$ be a Hopf matched pair, and let $(B,B^\ast)$ is a braided dual pair in a braided monoidal category $\C$. Then:
\begin{itemize}
\item[(1)]
The isomorphism $\Omega$ is a lax monoidal functor of prebraided monoidal categories.
\item[(2)]
If $\id_B$ and $\id_D$ are both convolution invertible, then $\Omega$ is an isomorphism of braided monoidal categories.
\end{itemize}
\end{theorem}

\begin{proof}
\begin{itemize}
\item[(1)]
For any object $(M,p_M,r_M),(N,p_N,r_N)\in{}_{B\times D}\YD(\C)^{B\times D}$,
we denote $$\Omega(M\otimes N,p_{M\otimes N},r_{M\otimes N})=(M\otimes N,p'_{M\otimes N},r'_{M\otimes N})$$
and define
a morphism
$J_{M,N}:\Omega(M)\otimes\Omega(N)\rightarrow\Omega(M\otimes N)$
in $\C$ as follows:
\begin{equation}\label{eqn:JMN}
J_{M,N}=
\ArxivFigure{0375}.
\end{equation}

It is a natural transformation in ${}_{D\times B^\ast}\YD(\C)^{D\times B^\ast}$, since one could check that $J_{M,N}$ preserves left $D\times 1$-actions, left $1\times B^\ast$-actions, right $D\times 1$-coactions and right $1\times B^\ast$-coactions, respectively.
For example, in order to find that $J_{M,N}$ preserves left $D\times 1$-actions, we need to show that
\begin{eqnarray*}
&&
J_{M,N}\circ (p'_M\otimes p'_N)
\circ(\id_{D\times B^\ast}\otimes c_{D\times B^\ast,M}\otimes\id_N)
\circ(\Delta_{D\times B^\ast}\otimes\id_{M\otimes N})  \\
&=&
p'_{M\otimes N}\circ(\id_{D\times B^\ast}\otimes J_{M,N})
\end{eqnarray*}
holds as a morphism from $(D\times 1)\otimes\Omega(M)\otimes\Omega(N)$
to $\Omega(M\otimes N)$.
In fact,
\begin{eqnarray*}
&&
\ArxivFigure{0376}
=
\ArxivFigure{0377}
=
\ArxivFigure{0378}
\\
&&=
\ArxivFigure{0379}
=
\ArxivFigure{0380}
=
\ArxivFigure{0381}
\\
&&=
\ArxivFigure{0382}
=
\ArxivFigure{0383}
=
\ArxivFigure{0384}
\\
&&=
\ArxivFigure{0385}
=
\ArxivFigure{0386}.
\end{eqnarray*}

Moreover, we could directly find that
$J_{\1,M}=\id_M=J_{M,\1}$,
and it remains to check that
$$J_{L\otimes M,N}\circ (J_{L,M}\otimes\id_{\Omega(N)})
=J_{L,M\otimes N}\circ (\id_{\Omega(L)}\otimes J_{M,N})$$
for all $(L,p_L,r_L),(M,p_M,r_M),(N,p_N,r_N)\in{}_{B\times D}\YD(\C)^{B\times D}$.
Specifically, we have following calculations:
\begin{eqnarray*}
&&
\ArxivFigure{0387}
=
\ArxivFigure{0388}
=
\ArxivFigure{0389}
=
\ArxivFigure{0390}
\\
&\overset{(\ref{eqn:YDrestricted2})}=&
\ArxivFigure{0391}
=
\ArxivFigure{0392}
=
\ArxivFigure{0393}
\\
&=&
\ArxivFigure{0394}.
\end{eqnarray*}

Now we define another morphism
$$I_{M,N}=
\ArxivFigure{0395},$$
which is natural in $M,N\in{}_{B\times D}\YD(\C)^{B\times D}$.
Then it is clear by Lemma \ref{lem:crossprodalg} that
\begin{equation}\label{eqn:JI=sigmac}
J_{M,N}\circ I_{M,N}=\sigma_{N,M}\circ c_{M,N}
\;\;\;\;\;\;\;\;(\forall M,N\in{}_{B\times D}\YD(\C)^{B\times D})
\end{equation}
holds.

On the other hand, let us denote the prebraiding of ${}_{D\times B^\ast}\YD(\C)^{D\times B^\ast}$ by $\sigma'$ as defined in Lemma \ref{lem:YDmodmonoidal}(2). Meanwhile, for any
$(M,p_M,r_M),(N,p_N,r_N)\in{}_{B\times D}\YD(\C)^{B\times D}$, suppose
$$\Omega(M,p_M,r_M)=(M,p'_M,r'_M)\;\;\;\;\text{and}\;\;\;\;
\Omega(N,p_N,r_N)=(N,p'_N,r'_N).$$
Then
we could know that
\begin{eqnarray*}
\sigma'_{\Omega(M),\Omega(N)}
&=&
\ArxivFigure{0396}
=
\ArxivFigure{0397}
=
\ArxivFigure{0398}
\\
&=&
I_{N,M}\circ c^{-1}_{N,M}\circ J_{M,N}.
\end{eqnarray*}
As a consequence,
\begin{eqnarray*}
&&
J_{N,M}\circ\sigma'_{\Omega(M),\Omega(N)}
=
J_{N,M}\circ I_{N,M}\circ c^{-1}_{N,M}\circ J_{M,N}  \\
&\overset{(\ref{eqn:JI=sigmac})}=&
\sigma_{M,N}\circ c_{N,M}\circ c^{-1}_{N,M}\circ J_{M,N}
=
\sigma_{M,N}\circ J_{M,N},
\end{eqnarray*}
which means that
the lax monoidal isomorphiam $(\Omega,J)$ preserves the prebraidings.

\item[(2)]
When $\id_B$ and $\id_D$ are both convolution invertible, it is known by Lemma \ref{lem:crossprod}(2) that the generalized smash biproduct $B\times D$ is a Hopf algebra.
Furthermore,
one could directly verify that
the natural transformation $J_{M,N}$ (\ref{eqn:JMN}) is an isomorphism with inverse
$$
\ArxivFigure{0399},$$
where $S_B$ is the convolution inverse of $\id_B$. Thus $J$ is a strong monoidal structure of $\Omega$.

\end{itemize}
\end{proof}

\section*{Acknowledgement}

The fundational results of this paper are established during a visit at University of Tsukuba, and the author is deeply grateful to Professor Akira Masuoka for motivations, guiding advice at the weekly seminars, as well as his considerate hospitality.
In addition, the author would like to thank Professors Shenglin Zhu and Gongxiang Liu for their valuable discussions.

ChatGPT was used to search for relavant notions and backgrounds in the literature, and to check the proofs in Section 5. The author takes sole responsibility for the contents of the paper.

\end{document}

%% file: figure-layout.tex
\UseRawInputEncoding

\usepackage{graphicx}
\newcommand{\ArxivImageBox}[5]{%
  \raisebox{-#4}[#3][#4]{%
    \makebox[#2][l]{\includegraphics[width=#2,height=#5]{figures/figure-#1.pdf}}%
  }%
}
\newcommand{\ArxivFigure}[1]{
    \relax\ifmmode
        {\vcenter{\hbox{
            \csname ArxivFigureData@#1\endcsname{}
        }}}
    \else
        \({\vcenter{\hbox{
            \csname ArxivFigureData@#1\endcsname{}
        }}}\)
    \fi
}

\expandafter\def\csname ArxivFigureData@0001\endcsname{\ArxivImageBox{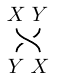}{1777170sp}{2541547sp}{0sp}{2541547sp}}
\expandafter\def\csname ArxivFigureData@0002\endcsname{\ArxivImageBox{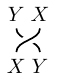}{1777170sp}{2541547sp}{0sp}{2541547sp}}
\expandafter\def\csname ArxivFigureData@0003\endcsname{\ArxivImageBox{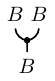}{1712772sp}{2541547sp}{0sp}{2541547sp}}
\expandafter\def\csname ArxivFigureData@0004\endcsname{\ArxivImageBox{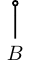}{966836sp}{2117528sp}{0sp}{2117528sp}}
\expandafter\def\csname ArxivFigureData@0005\endcsname{\ArxivImageBox{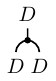}{1743584sp}{2541547sp}{0sp}{2541547sp}}
\expandafter\def\csname ArxivFigureData@0006\endcsname{\ArxivImageBox{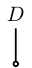}{997648sp}{2117552sp}{0sp}{2117552sp}}
\expandafter\def\csname ArxivFigureData@0007\endcsname{\ArxivImageBox{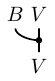}{1711748sp}{2541547sp}{0sp}{2541547sp}}
\expandafter\def\csname ArxivFigureData@0008\endcsname{\ArxivImageBox{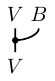}{1711747sp}{2541547sp}{0sp}{2541547sp}}
\expandafter\def\csname ArxivFigureData@0009\endcsname{\ArxivImageBox{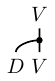}{1727154sp}{2541547sp}{0sp}{2541547sp}}
\expandafter\def\csname ArxivFigureData@0010\endcsname{\ArxivImageBox{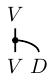}{1727153sp}{2541547sp}{0sp}{2541547sp}}
\expandafter\def\csname ArxivFigureData@0011\endcsname{\ArxivImageBox{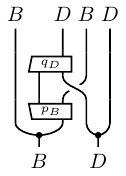}{3976700sp}{5547667sp}{0sp}{5547667sp}}
\expandafter\def\csname ArxivFigureData@0012\endcsname{\ArxivImageBox{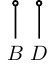}{1737952sp}{2127328sp}{0sp}{2127328sp}}
\expandafter\def\csname ArxivFigureData@0013\endcsname{\ArxivImageBox{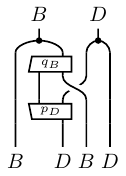}{3976600sp}{5548715sp}{0sp}{5548715sp}}
\expandafter\def\csname ArxivFigureData@0014\endcsname{\ArxivImageBox{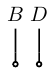}{1737977sp}{2127328sp}{0sp}{2127328sp}}
\expandafter\def\csname ArxivFigureData@0015\endcsname{\ArxivImageBox{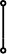}{201608sp}{1693794sp}{0sp}{1693794sp}}
\expandafter\def\csname ArxivFigureData@0016\endcsname{\ArxivImageBox{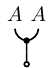}{1684020sp}{2127352sp}{0sp}{2127352sp}}
\expandafter\def\csname ArxivFigureData@0017\endcsname{\ArxivImageBox{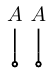}{1684115sp}{2127328sp}{0sp}{2127328sp}}
\expandafter\def\csname ArxivFigureData@0018\endcsname{\ArxivImageBox{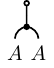}{1684020sp}{2127256sp}{0sp}{2127256sp}}
\expandafter\def\csname ArxivFigureData@0019\endcsname{\ArxivImageBox{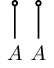}{1684090sp}{2127328sp}{0sp}{2127328sp}}
\expandafter\def\csname ArxivFigureData@0020\endcsname{\ArxivImageBox{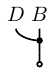}{1737882sp}{2127352sp}{0sp}{2127352sp}}
\expandafter\def\csname ArxivFigureData@0021\endcsname{\ArxivImageBox{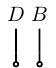}{1737977sp}{2127328sp}{0sp}{2127328sp}}
\expandafter\def\csname ArxivFigureData@0022\endcsname{\ArxivImageBox{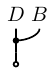}{1737881sp}{2127352sp}{0sp}{2127352sp}}
\expandafter\def\csname ArxivFigureData@0023\endcsname{\ArxivImageBox{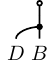}{1737882sp}{2127256sp}{0sp}{2127256sp}}
\expandafter\def\csname ArxivFigureData@0024\endcsname{\ArxivImageBox{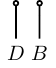}{1737952sp}{2127328sp}{0sp}{2127328sp}}
\expandafter\def\csname ArxivFigureData@0025\endcsname{\ArxivImageBox{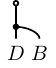}{1737881sp}{2127304sp}{0sp}{2127304sp}}
\expandafter\def\csname ArxivFigureData@0026\endcsname{\ArxivImageBox{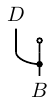}{1737882sp}{3306888sp}{0sp}{3306888sp}}
\expandafter\def\csname ArxivFigureData@0027\endcsname{\ArxivImageBox{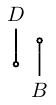}{1737894sp}{3306936sp}{0sp}{3306936sp}}
\expandafter\def\csname ArxivFigureData@0028\endcsname{\ArxivImageBox{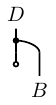}{1737943sp}{3306984sp}{0sp}{3306984sp}}
\expandafter\def\csname ArxivFigureData@0029\endcsname{\ArxivImageBox{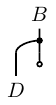}{1737894sp}{3306912sp}{0sp}{3306912sp}}
\expandafter\def\csname ArxivFigureData@0030\endcsname{\ArxivImageBox{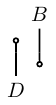}{1737943sp}{3306864sp}{0sp}{3306864sp}}
\expandafter\def\csname ArxivFigureData@0031\endcsname{\ArxivImageBox{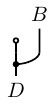}{1737980sp}{3306936sp}{0sp}{3306936sp}}
\expandafter\def\csname ArxivFigureData@0032\endcsname{\ArxivImageBox{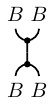}{1722476sp}{3306936sp}{0sp}{3306936sp}}
\expandafter\def\csname ArxivFigureData@0033\endcsname{\ArxivImageBox{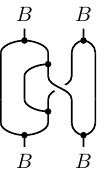}{3125433sp}{5547911sp}{0sp}{5547911sp}}
\expandafter\def\csname ArxivFigureData@0034\endcsname{\ArxivImageBox{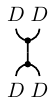}{1753288sp}{3306936sp}{0sp}{3306936sp}}
\expandafter\def\csname ArxivFigureData@0035\endcsname{\ArxivImageBox{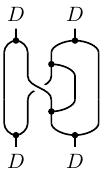}{3141583sp}{5547911sp}{0sp}{5547911sp}}
\expandafter\def\csname ArxivFigureData@0036\endcsname{\ArxivImageBox{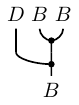}{2483912sp}{3306936sp}{0sp}{3306936sp}}
\expandafter\def\csname ArxivFigureData@0037\endcsname{\ArxivImageBox{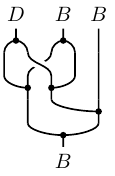}{3603068sp}{5547021sp}{0sp}{5547021sp}}
\expandafter\def\csname ArxivFigureData@0038\endcsname{\ArxivImageBox{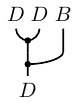}{2483916sp}{3306984sp}{0sp}{3306984sp}}
\expandafter\def\csname ArxivFigureData@0039\endcsname{\ArxivImageBox{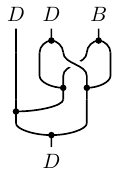}{3602991sp}{5546821sp}{0sp}{5546821sp}}
\expandafter\def\csname ArxivFigureData@0040\endcsname{\ArxivImageBox{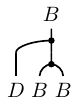}{2483924sp}{3307032sp}{0sp}{3307032sp}}
\expandafter\def\csname ArxivFigureData@0041\endcsname{\ArxivImageBox{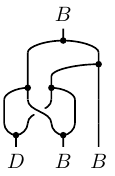}{3602956sp}{5547621sp}{0sp}{5547621sp}}
\expandafter\def\csname ArxivFigureData@0042\endcsname{\ArxivImageBox{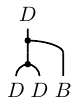}{2483854sp}{3306984sp}{0sp}{3306984sp}}
\expandafter\def\csname ArxivFigureData@0043\endcsname{\ArxivImageBox{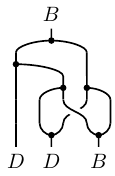}{3603003sp}{5547821sp}{0sp}{5547821sp}}
\expandafter\def\csname ArxivFigureData@0044\endcsname{\ArxivImageBox{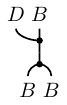}{2110867sp}{3306936sp}{0sp}{3306936sp}}
\expandafter\def\csname ArxivFigureData@0045\endcsname{\ArxivImageBox{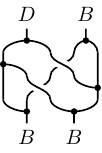}{3200053sp}{4800214sp}{0sp}{4800214sp}}
\expandafter\def\csname ArxivFigureData@0046\endcsname{\ArxivImageBox{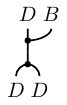}{2110827sp}{3306936sp}{0sp}{3306936sp}}
\expandafter\def\csname ArxivFigureData@0047\endcsname{\ArxivImageBox{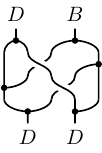}{3215614sp}{4800266sp}{0sp}{4800266sp}}
\expandafter\def\csname ArxivFigureData@0048\endcsname{\ArxivImageBox{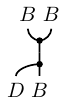}{2110867sp}{3306984sp}{0sp}{3306984sp}}
\expandafter\def\csname ArxivFigureData@0049\endcsname{\ArxivImageBox{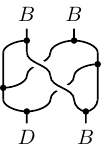}{3200053sp}{4800118sp}{0sp}{4800118sp}}
\expandafter\def\csname ArxivFigureData@0050\endcsname{\ArxivImageBox{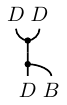}{2110827sp}{3306984sp}{0sp}{3306984sp}}
\expandafter\def\csname ArxivFigureData@0051\endcsname{\ArxivImageBox{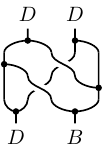}{3215614sp}{4800066sp}{0sp}{4800066sp}}
\expandafter\def\csname ArxivFigureData@0052\endcsname{\ArxivImageBox{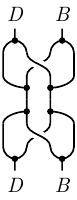}{2416579sp}{6294111sp}{0sp}{6294111sp}}
\expandafter\def\csname ArxivFigureData@0053\endcsname{\ArxivImageBox{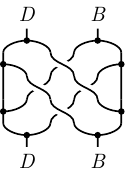}{3931732sp}{5547273sp}{0sp}{5547273sp}}
\expandafter\def\csname ArxivFigureData@0054\endcsname{\ArxivImageBox{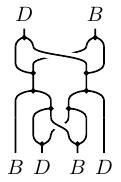}{3789394sp}{5734598sp}{0sp}{5734598sp}}
\expandafter\def\csname ArxivFigureData@0055\endcsname{\ArxivImageBox{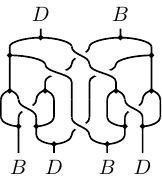}{5089390sp}{5736491sp}{0sp}{5736491sp}}
\expandafter\def\csname ArxivFigureData@0056\endcsname{\ArxivImageBox{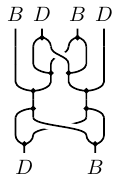}{3789456sp}{5733743sp}{0sp}{5733743sp}}
\expandafter\def\csname ArxivFigureData@0057\endcsname{\ArxivImageBox{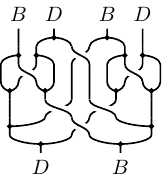}{5089390sp}{5735612sp}{0sp}{5735612sp}}
\expandafter\def\csname ArxivFigureData@0058\endcsname{\ArxivImageBox{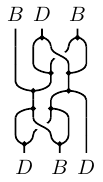}{3229849sp}{5734219sp}{0sp}{5734219sp}}
\expandafter\def\csname ArxivFigureData@0059\endcsname{\ArxivImageBox{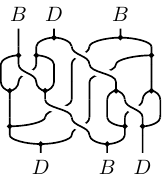}{5089048sp}{5736065sp}{0sp}{5736065sp}}
\expandafter\def\csname ArxivFigureData@0060\endcsname{\ArxivImageBox{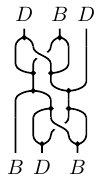}{3229911sp}{5734268sp}{0sp}{5734268sp}}
\expandafter\def\csname ArxivFigureData@0061\endcsname{\ArxivImageBox{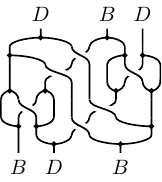}{5089048sp}{5736038sp}{0sp}{5736038sp}}
\expandafter\def\csname ArxivFigureData@0062\endcsname{\ArxivImageBox{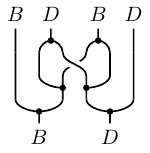}{4722215sp}{4800534sp}{0sp}{4800534sp}}
\expandafter\def\csname ArxivFigureData@0063\endcsname{\ArxivImageBox{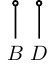}{1737952sp}{2127328sp}{0sp}{2127328sp}}
\expandafter\def\csname ArxivFigureData@0064\endcsname{\ArxivImageBox{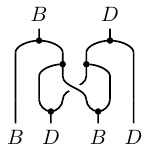}{4722090sp}{4801056sp}{0sp}{4801056sp}}
\expandafter\def\csname ArxivFigureData@0065\endcsname{\ArxivImageBox{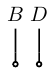}{1737977sp}{2127328sp}{0sp}{2127328sp}}
\expandafter\def\csname ArxivFigureData@0066\endcsname{\ArxivImageBox{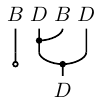}{3229993sp}{3307008sp}{0sp}{3307008sp}}
\expandafter\def\csname ArxivFigureData@0067\endcsname{\ArxivImageBox{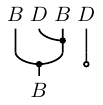}{3230219sp}{3306936sp}{0sp}{3306936sp}}
\expandafter\def\csname ArxivFigureData@0068\endcsname{\ArxivImageBox{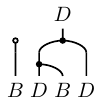}{3229918sp}{3307128sp}{0sp}{3307128sp}}
\expandafter\def\csname ArxivFigureData@0069\endcsname{\ArxivImageBox{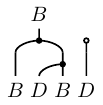}{3230144sp}{3307128sp}{0sp}{3307128sp}}
\expandafter\def\csname ArxivFigureData@0070\endcsname{\ArxivImageBox{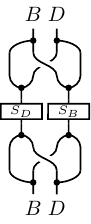}{2842380sp}{7041348sp}{0sp}{7041348sp}}
\expandafter\def\csname ArxivFigureData@0071\endcsname{\ArxivImageBox{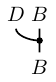}{1737882sp}{2560955sp}{0sp}{2560955sp}}
\expandafter\def\csname ArxivFigureData@0072\endcsname{\ArxivImageBox{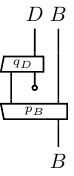}{2331108sp}{5546187sp}{0sp}{5546187sp}}
\expandafter\def\csname ArxivFigureData@0073\endcsname{\ArxivImageBox{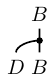}{1737882sp}{2560955sp}{0sp}{2560955sp}}
\expandafter\def\csname ArxivFigureData@0074\endcsname{\ArxivImageBox{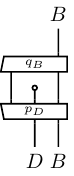}{2331046sp}{5546491sp}{0sp}{5546491sp}}
\expandafter\def\csname ArxivFigureData@0075\endcsname{\ArxivImageBox{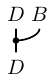}{1737881sp}{2560955sp}{0sp}{2560955sp}}
\expandafter\def\csname ArxivFigureData@0076\endcsname{\ArxivImageBox{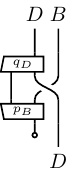}{2346354sp}{5546583sp}{0sp}{5546583sp}}
\expandafter\def\csname ArxivFigureData@0077\endcsname{\ArxivImageBox{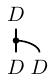}{1753287sp}{2560955sp}{0sp}{2560955sp}}
\expandafter\def\csname ArxivFigureData@0078\endcsname{\ArxivImageBox{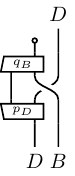}{2346416sp}{5546683sp}{0sp}{5546683sp}}
\expandafter\def\csname ArxivFigureData@0079\endcsname{\ArxivImageBox{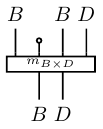}{3230317sp}{4053249sp}{0sp}{4053249sp}}
\expandafter\def\csname ArxivFigureData@0080\endcsname{\ArxivImageBox{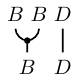}{2483916sp}{2560955sp}{0sp}{2560955sp}}
\expandafter\def\csname ArxivFigureData@0081\endcsname{\ArxivImageBox{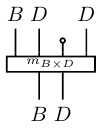}{3230269sp}{4053249sp}{0sp}{4053249sp}}
\expandafter\def\csname ArxivFigureData@0082\endcsname{\ArxivImageBox{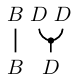}{2483924sp}{2560955sp}{0sp}{2560955sp}}
\expandafter\def\csname ArxivFigureData@0083\endcsname{\ArxivImageBox{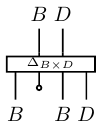}{3230242sp}{4053499sp}{0sp}{4053499sp}}
\expandafter\def\csname ArxivFigureData@0084\endcsname{\ArxivImageBox{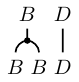}{2483916sp}{2560955sp}{0sp}{2560955sp}}
\expandafter\def\csname ArxivFigureData@0085\endcsname{\ArxivImageBox{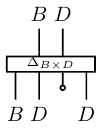}{3230194sp}{4053499sp}{0sp}{4053499sp}}
\expandafter\def\csname ArxivFigureData@0086\endcsname{\ArxivImageBox{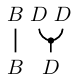}{2483924sp}{2560955sp}{0sp}{2560955sp}}
\expandafter\def\csname ArxivFigureData@0087\endcsname{\ArxivImageBox{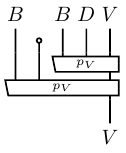}{3960294sp}{4799496sp}{0sp}{4799496sp}}
\expandafter\def\csname ArxivFigureData@0088\endcsname{\ArxivImageBox{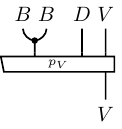}{3822462sp}{4053097sp}{0sp}{4053097sp}}
\expandafter\def\csname ArxivFigureData@0089\endcsname{\ArxivImageBox{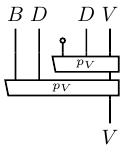}{3960246sp}{4799496sp}{0sp}{4799496sp}}
\expandafter\def\csname ArxivFigureData@0090\endcsname{\ArxivImageBox{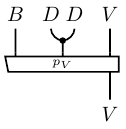}{3959836sp}{4053097sp}{0sp}{4053097sp}}
\expandafter\def\csname ArxivFigureData@0091\endcsname{\ArxivImageBox{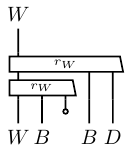}{4066575sp}{4800666sp}{0sp}{4800666sp}}
\expandafter\def\csname ArxivFigureData@0092\endcsname{\ArxivImageBox{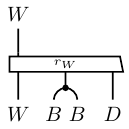}{4066165sp}{4053499sp}{0sp}{4053499sp}}
\expandafter\def\csname ArxivFigureData@0093\endcsname{\ArxivImageBox{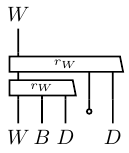}{4066523sp}{4800666sp}{0sp}{4800666sp}}
\expandafter\def\csname ArxivFigureData@0094\endcsname{\ArxivImageBox{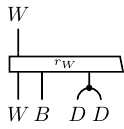}{3912596sp}{4053499sp}{0sp}{4053499sp}}
\expandafter\def\csname ArxivFigureData@0095\endcsname{\ArxivImageBox{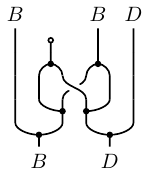}{4713205sp}{5528509sp}{0sp}{5528509sp}}
\expandafter\def\csname ArxivFigureData@0096\endcsname{\ArxivImageBox{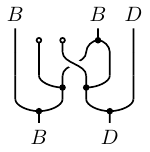}{4712661sp}{4781126sp}{0sp}{4781126sp}}
\expandafter\def\csname ArxivFigureData@0097\endcsname{\ArxivImageBox{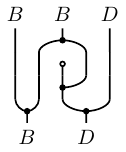}{3966302sp}{4781176sp}{0sp}{4781176sp}}
\expandafter\def\csname ArxivFigureData@0098\endcsname{\ArxivImageBox{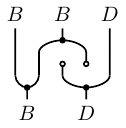}{3966302sp}{4034149sp}{0sp}{4034149sp}}
\expandafter\def\csname ArxivFigureData@0099\endcsname{\ArxivImageBox{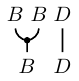}{2474212sp}{2541547sp}{0sp}{2541547sp}}
\expandafter\def\csname ArxivFigureData@0100\endcsname{\ArxivImageBox{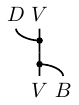}{2483817sp}{3306936sp}{0sp}{3306936sp}}
\expandafter\def\csname ArxivFigureData@0101\endcsname{\ArxivImageBox{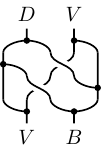}{3185764sp}{4800214sp}{0sp}{4800214sp}}
\expandafter\def\csname ArxivFigureData@0102\endcsname{\ArxivImageBox{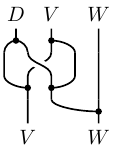}{3693089sp}{4800460sp}{0sp}{4800460sp}}
\expandafter\def\csname ArxivFigureData@0103\endcsname{\ArxivImageBox{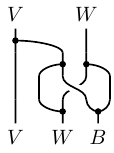}{3586858sp}{4800760sp}{0sp}{4800760sp}}
\expandafter\def\csname ArxivFigureData@0104\endcsname{\ArxivImageBox{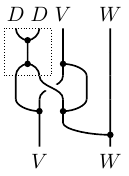}{4056627sp}{5528313sp}{0sp}{5528313sp}}
\expandafter\def\csname ArxivFigureData@0105\endcsname{\ArxivImageBox{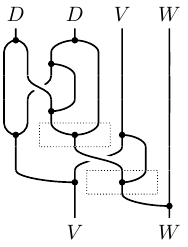}{5922290sp}{7775774sp}{0sp}{7775774sp}}
\expandafter\def\csname ArxivFigureData@0106\endcsname{\ArxivImageBox{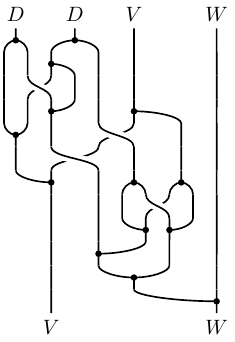}{7414597sp}{10784187sp}{0sp}{10784187sp}}
\expandafter\def\csname ArxivFigureData@0107\endcsname{\ArxivImageBox{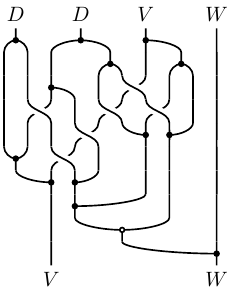}{7414324sp}{9279679sp}{0sp}{9279679sp}}
\expandafter\def\csname ArxivFigureData@0108\endcsname{\ArxivImageBox{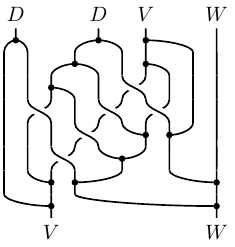}{7414062sp}{7777750sp}{0sp}{7777750sp}}
\expandafter\def\csname ArxivFigureData@0109\endcsname{\ArxivImageBox{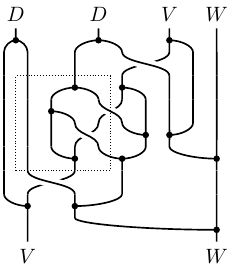}{7413913sp}{8527229sp}{0sp}{8527229sp}}
\expandafter\def\csname ArxivFigureData@0110\endcsname{\ArxivImageBox{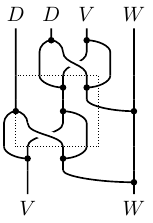}{4802750sp}{7025223sp}{0sp}{7025223sp}}
\expandafter\def\csname ArxivFigureData@0111\endcsname{\ArxivImageBox{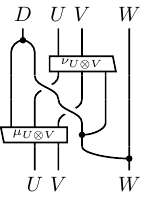}{4654956sp}{6277137sp}{0sp}{6277137sp}}
\expandafter\def\csname ArxivFigureData@0112\endcsname{\ArxivImageBox{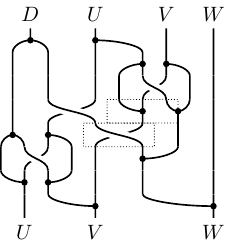}{7315176sp}{7779650sp}{0sp}{7779650sp}}
\expandafter\def\csname ArxivFigureData@0113\endcsname{\ArxivImageBox{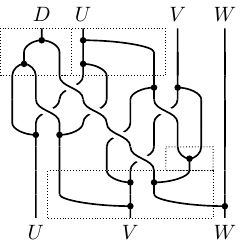}{7674859sp}{7780292sp}{0sp}{7780292sp}}
\expandafter\def\csname ArxivFigureData@0114\endcsname{\ArxivImageBox{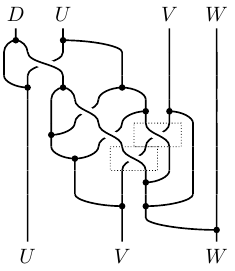}{7414860sp}{8530386sp}{0sp}{8530386sp}}
\expandafter\def\csname ArxivFigureData@0115\endcsname{\ArxivImageBox{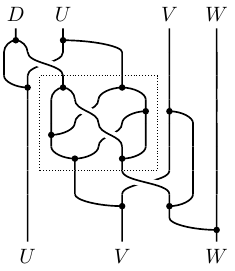}{7414860sp}{8530636sp}{0sp}{8530636sp}}
\expandafter\def\csname ArxivFigureData@0116\endcsname{\ArxivImageBox{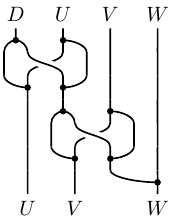}{5549271sp}{7027391sp}{0sp}{7027391sp}}
\expandafter\def\csname ArxivFigureData@0117\endcsname{\ArxivImageBox{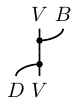}{2483817sp}{3306936sp}{0sp}{3306936sp}}
\expandafter\def\csname ArxivFigureData@0118\endcsname{\ArxivImageBox{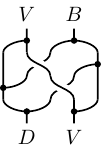}{3185764sp}{4800314sp}{0sp}{4800314sp}}
\expandafter\def\csname ArxivFigureData@0119\endcsname{\ArxivImageBox{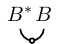}{1873118sp}{1384732sp}{0sp}{1384732sp}}
\expandafter\def\csname ArxivFigureData@0120\endcsname{\ArxivImageBox{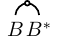}{1873118sp}{1384719sp}{0sp}{1384719sp}}
\expandafter\def\csname ArxivFigureData@0121\endcsname{\ArxivImageBox{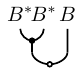}{2619152sp}{2130796sp}{0sp}{2130796sp}}
\expandafter\def\csname ArxivFigureData@0122\endcsname{\ArxivImageBox{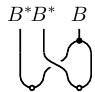}{2992338sp}{2877188sp}{0sp}{2877188sp}}
\expandafter\def\csname ArxivFigureData@0123\endcsname{\ArxivImageBox{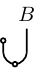}{1335052sp}{2127256sp}{0sp}{2127256sp}}
\expandafter\def\csname ArxivFigureData@0124\endcsname{\ArxivImageBox{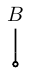}{976540sp}{2127256sp}{0sp}{2127256sp}}
\expandafter\def\csname ArxivFigureData@0125\endcsname{\ArxivImageBox{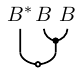}{2619148sp}{2130700sp}{0sp}{2130700sp}}
\expandafter\def\csname ArxivFigureData@0126\endcsname{\ArxivImageBox{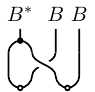}{2992465sp}{2877236sp}{0sp}{2877236sp}}
\expandafter\def\csname ArxivFigureData@0127\endcsname{\ArxivImageBox{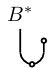}{1485777sp}{2130700sp}{0sp}{2130700sp}}
\expandafter\def\csname ArxivFigureData@0128\endcsname{\ArxivImageBox{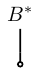}{1277824sp}{2130700sp}{0sp}{2130700sp}}
\expandafter\def\csname ArxivFigureData@0129\endcsname{\ArxivImageBox{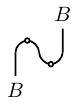}{2468522sp}{3306864sp}{0sp}{3306864sp}}
\expandafter\def\csname ArxivFigureData@0130\endcsname{\ArxivImageBox{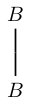}{976552sp}{3306864sp}{0sp}{3306864sp}}
\expandafter\def\csname ArxivFigureData@0131\endcsname{\ArxivImageBox{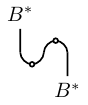}{2769732sp}{3313824sp}{0sp}{3313824sp}}
\expandafter\def\csname ArxivFigureData@0132\endcsname{\ArxivImageBox{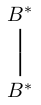}{1277836sp}{3313752sp}{0sp}{3313752sp}}
\expandafter\def\csname ArxivFigureData@0133\endcsname{\ArxivImageBox{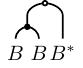}{2619090sp}{2130748sp}{0sp}{2130748sp}}
\expandafter\def\csname ArxivFigureData@0134\endcsname{\ArxivImageBox{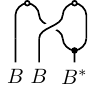}{2992350sp}{2877403sp}{0sp}{2877403sp}}
\expandafter\def\csname ArxivFigureData@0135\endcsname{\ArxivImageBox{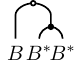}{2619160sp}{2130796sp}{0sp}{2130796sp}}
\expandafter\def\csname ArxivFigureData@0136\endcsname{\ArxivImageBox{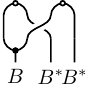}{2992378sp}{2877303sp}{0sp}{2877303sp}}
\expandafter\def\csname ArxivFigureData@0137\endcsname{\ArxivImageBox{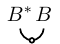}{1863414sp}{1375028sp}{0sp}{1375028sp}}
\expandafter\def\csname ArxivFigureData@0138\endcsname{\ArxivImageBox{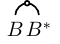}{1863414sp}{1375015sp}{0sp}{1375015sp}}
\expandafter\def\csname ArxivFigureData@0139\endcsname{\ArxivImageBox{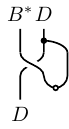}{2157195sp}{4056489sp}{0sp}{4056489sp}}
\expandafter\def\csname ArxivFigureData@0140\endcsname{\ArxivImageBox{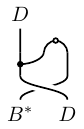}{2634459sp}{4056587sp}{0sp}{4056587sp}}
\expandafter\def\csname ArxivFigureData@0141\endcsname{\ArxivImageBox{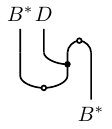}{3515692sp}{4060227sp}{0sp}{4060227sp}}
\expandafter\def\csname ArxivFigureData@0142\endcsname{\ArxivImageBox{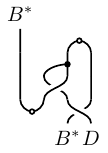}{3380494sp}{4806980sp}{0sp}{4806980sp}}
\expandafter\def\csname ArxivFigureData@0143\endcsname{\ArxivImageBox{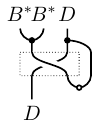}{2898401sp}{4037277sp}{0sp}{4037277sp}}
\expandafter\def\csname ArxivFigureData@0144\endcsname{\ArxivImageBox{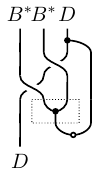}{2898663sp}{5530711sp}{0sp}{5530711sp}}
\expandafter\def\csname ArxivFigureData@0145\endcsname{\ArxivImageBox{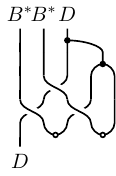}{3644935sp}{5530933sp}{0sp}{5530933sp}}
\expandafter\def\csname ArxivFigureData@0146\endcsname{\ArxivImageBox{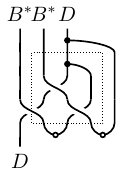}{3644935sp}{5530981sp}{0sp}{5530981sp}}
\expandafter\def\csname ArxivFigureData@0147\endcsname{\ArxivImageBox{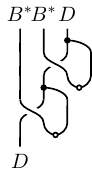}{2898487sp}{5530565sp}{0sp}{5530565sp}}
\expandafter\def\csname ArxivFigureData@0148\endcsname{\ArxivImageBox{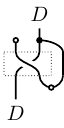}{2017119sp}{4033685sp}{0sp}{4033685sp}}
\expandafter\def\csname ArxivFigureData@0149\endcsname{\ArxivImageBox{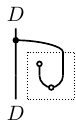}{2377136sp}{4033781sp}{0sp}{4033781sp}}
\expandafter\def\csname ArxivFigureData@0150\endcsname{\ArxivImageBox{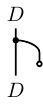}{1345738sp}{3287504sp}{0sp}{3287504sp}}
\expandafter\def\csname ArxivFigureData@0151\endcsname{\ArxivImageBox{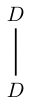}{997660sp}{3287456sp}{0sp}{3287456sp}}
\expandafter\def\csname ArxivFigureData@0152\endcsname{\ArxivImageBox{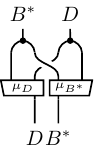}{2939350sp}{4806878sp}{0sp}{4806878sp}}
\expandafter\def\csname ArxivFigureData@0153\endcsname{\ArxivImageBox{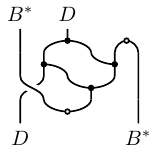}{5007932sp}{4807420sp}{0sp}{4807420sp}}
\expandafter\def\csname ArxivFigureData@0154\endcsname{\ArxivImageBox{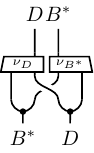}{2939250sp}{4806590sp}{0sp}{4806590sp}}
\expandafter\def\csname ArxivFigureData@0155\endcsname{\ArxivImageBox{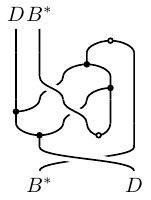}{4737078sp}{6301814sp}{0sp}{6301814sp}}
\expandafter\def\csname ArxivFigureData@0156\endcsname{\ArxivImageBox{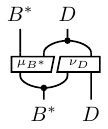}{3380456sp}{4059959sp}{0sp}{4059959sp}}
\expandafter\def\csname ArxivFigureData@0157\endcsname{\ArxivImageBox{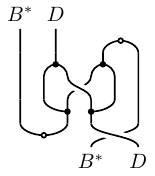}{4872414sp}{5554805sp}{0sp}{5554805sp}}
\expandafter\def\csname ArxivFigureData@0158\endcsname{\ArxivImageBox{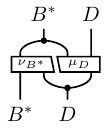}{3380555sp}{4059957sp}{0sp}{4059957sp}}
\expandafter\def\csname ArxivFigureData@0159\endcsname{\ArxivImageBox{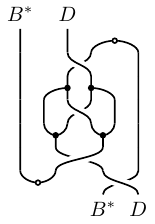}{4872562sp}{7051045sp}{0sp}{7051045sp}}
\expandafter\def\csname ArxivFigureData@0160\endcsname{\ArxivImageBox{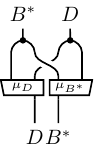}{2939350sp}{4787470sp}{0sp}{4787470sp}}
\expandafter\def\csname ArxivFigureData@0161\endcsname{\ArxivImageBox{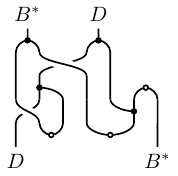}{5608989sp}{5536319sp}{0sp}{5536319sp}}
\expandafter\def\csname ArxivFigureData@0162\endcsname{\ArxivImageBox{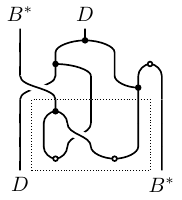}{5744213sp}{6284876sp}{0sp}{6284876sp}}
\expandafter\def\csname ArxivFigureData@0163\endcsname{\ArxivImageBox{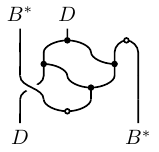}{4998228sp}{4788012sp}{0sp}{4788012sp}}
\expandafter\def\csname ArxivFigureData@0164\endcsname{\ArxivImageBox{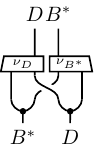}{2939250sp}{4787182sp}{0sp}{4787182sp}}
\expandafter\def\csname ArxivFigureData@0165\endcsname{\ArxivImageBox{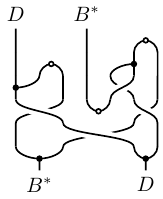}{5132196sp}{6282258sp}{0sp}{6282258sp}}
\expandafter\def\csname ArxivFigureData@0166\endcsname{\ArxivImageBox{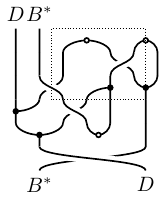}{5100343sp}{6282848sp}{0sp}{6282848sp}}
\expandafter\def\csname ArxivFigureData@0167\endcsname{\ArxivImageBox{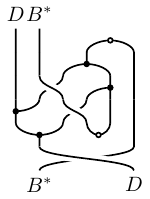}{4727374sp}{6282406sp}{0sp}{6282406sp}}
\expandafter\def\csname ArxivFigureData@0168\endcsname{\ArxivImageBox{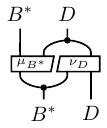}{3370752sp}{4040551sp}{0sp}{4040551sp}}
\expandafter\def\csname ArxivFigureData@0169\endcsname{\ArxivImageBox{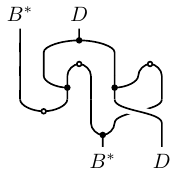}{5609107sp}{5535099sp}{0sp}{5535099sp}}
\expandafter\def\csname ArxivFigureData@0170\endcsname{\ArxivImageBox{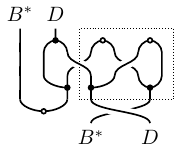}{5496413sp}{4787892sp}{0sp}{4787892sp}}
\expandafter\def\csname ArxivFigureData@0171\endcsname{\ArxivImageBox{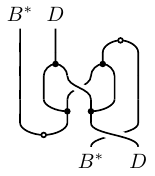}{4862710sp}{5535397sp}{0sp}{5535397sp}}
\expandafter\def\csname ArxivFigureData@0172\endcsname{\ArxivImageBox{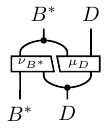}{3370851sp}{4040549sp}{0sp}{4040549sp}}
\expandafter\def\csname ArxivFigureData@0173\endcsname{\ArxivImageBox{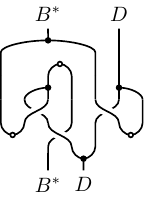}{4529635sp}{6283126sp}{0sp}{6283126sp}}
\expandafter\def\csname ArxivFigureData@0174\endcsname{\ArxivImageBox{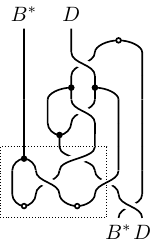}{4987915sp}{7781493sp}{0sp}{7781493sp}}
\expandafter\def\csname ArxivFigureData@0175\endcsname{\ArxivImageBox{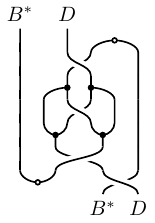}{4862858sp}{7031637sp}{0sp}{7031637sp}}
\expandafter\def\csname ArxivFigureData@0176\endcsname{\ArxivImageBox{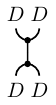}{1743584sp}{3287528sp}{0sp}{3287528sp}}
\expandafter\def\csname ArxivFigureData@0177\endcsname{\ArxivImageBox{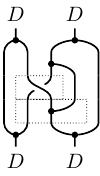}{3136731sp}{5528503sp}{0sp}{5528503sp}}
\expandafter\def\csname ArxivFigureData@0178\endcsname{\ArxivImageBox{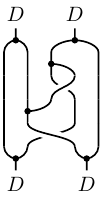}{3235705sp}{6276113sp}{0sp}{6276113sp}}
\expandafter\def\csname ArxivFigureData@0179\endcsname{\ArxivImageBox{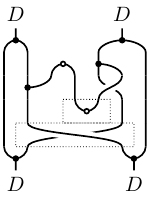}{4727675sp}{6276735sp}{0sp}{6276735sp}}
\expandafter\def\csname ArxivFigureData@0180\endcsname{\ArxivImageBox{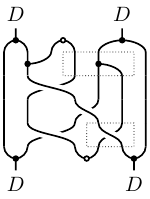}{4727899sp}{6277407sp}{0sp}{6277407sp}}
\expandafter\def\csname ArxivFigureData@0181\endcsname{\ArxivImageBox{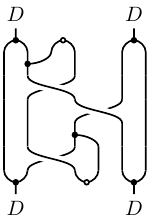}{4727925sp}{7026340sp}{0sp}{7026340sp}}
\expandafter\def\csname ArxivFigureData@0182\endcsname{\ArxivImageBox{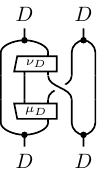}{3135987sp}{5528307sp}{0sp}{5528307sp}}
\expandafter\def\csname ArxivFigureData@0183\endcsname{\ArxivImageBox{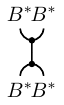}{2014056sp}{3294416sp}{0sp}{3294416sp}}
\expandafter\def\csname ArxivFigureData@0184\endcsname{\ArxivImageBox{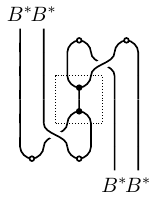}{4998604sp}{6284648sp}{0sp}{6284648sp}}
\expandafter\def\csname ArxivFigureData@0185\endcsname{\ArxivImageBox{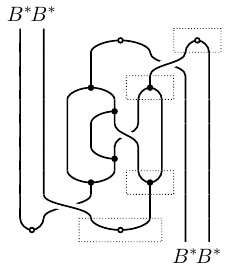}{7237360sp}{8538920sp}{0sp}{8538920sp}}
\expandafter\def\csname ArxivFigureData@0186\endcsname{\ArxivImageBox{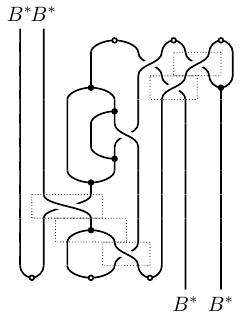}{7610616sp}{10047843sp}{0sp}{10047843sp}}
\expandafter\def\csname ArxivFigureData@0187\endcsname{\ArxivImageBox{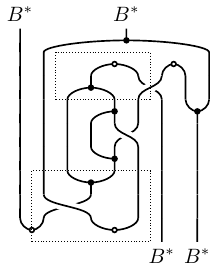}{6864280sp}{8538526sp}{0sp}{8538526sp}}
\expandafter\def\csname ArxivFigureData@0188\endcsname{\ArxivImageBox{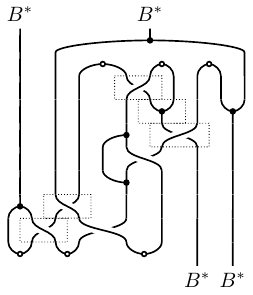}{7984336sp}{9296776sp}{0sp}{9296776sp}}
\expandafter\def\csname ArxivFigureData@0189\endcsname{\ArxivImageBox{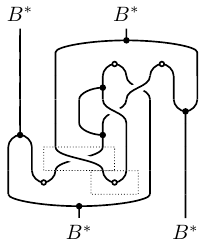}{6490167sp}{7785016sp}{0sp}{7785016sp}}
\expandafter\def\csname ArxivFigureData@0190\endcsname{\ArxivImageBox{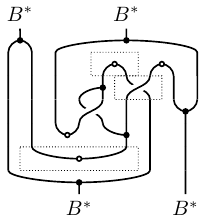}{6490189sp}{7034331sp}{0sp}{7034331sp}}
\expandafter\def\csname ArxivFigureData@0191\endcsname{\ArxivImageBox{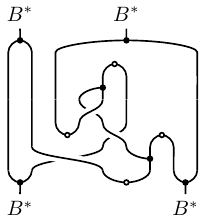}{6490417sp}{7034658sp}{0sp}{7034658sp}}
\expandafter\def\csname ArxivFigureData@0192\endcsname{\ArxivImageBox{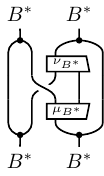}{3271967sp}{5535195sp}{0sp}{5535195sp}}
\expandafter\def\csname ArxivFigureData@0193\endcsname{\ArxivImageBox{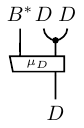}{2624850sp}{4037085sp}{0sp}{4037085sp}}
\expandafter\def\csname ArxivFigureData@0194\endcsname{\ArxivImageBox{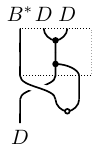}{2885018sp}{4783536sp}{0sp}{4783536sp}}
\expandafter\def\csname ArxivFigureData@0195\endcsname{\ArxivImageBox{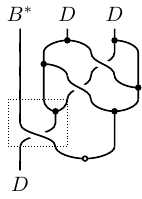}{4465226sp}{6277967sp}{0sp}{6277967sp}}
\expandafter\def\csname ArxivFigureData@0196\endcsname{\ArxivImageBox{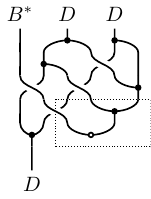}{4750006sp}{6277667sp}{0sp}{6277667sp}}
\expandafter\def\csname ArxivFigureData@0197\endcsname{\ArxivImageBox{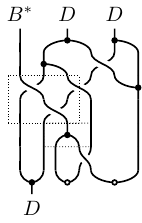}{4465008sp}{7029068sp}{0sp}{7029068sp}}
\expandafter\def\csname ArxivFigureData@0198\endcsname{\ArxivImageBox{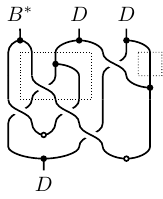}{5123918sp}{6279580sp}{0sp}{6279580sp}}
\expandafter\def\csname ArxivFigureData@0199\endcsname{\ArxivImageBox{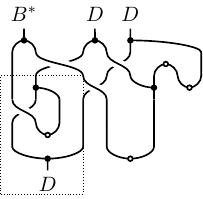}{6381367sp}{6280481sp}{0sp}{6280481sp}}
\expandafter\def\csname ArxivFigureData@0200\endcsname{\ArxivImageBox{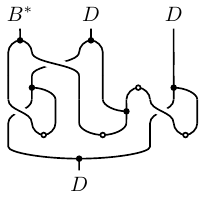}{6257342sp}{6280483sp}{0sp}{6280483sp}}
\expandafter\def\csname ArxivFigureData@0201\endcsname{\ArxivImageBox{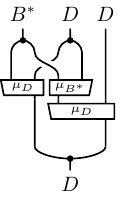}{3833610sp}{6278392sp}{0sp}{6278392sp}}
\expandafter\def\csname ArxivFigureData@0202\endcsname{\ArxivImageBox{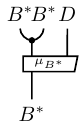}{2624854sp}{4040625sp}{0sp}{4040625sp}}
\expandafter\def\csname ArxivFigureData@0203\endcsname{\ArxivImageBox{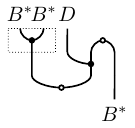}{4251984sp}{4041063sp}{0sp}{4041063sp}}
\expandafter\def\csname ArxivFigureData@0204\endcsname{\ArxivImageBox{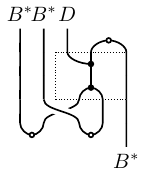}{4625212sp}{5535695sp}{0sp}{5535695sp}}
\expandafter\def\csname ArxivFigureData@0205\endcsname{\ArxivImageBox{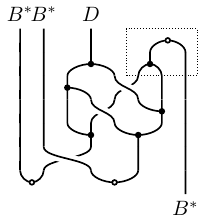}{6490452sp}{7034189sp}{0sp}{7034189sp}}
\expandafter\def\csname ArxivFigureData@0206\endcsname{\ArxivImageBox{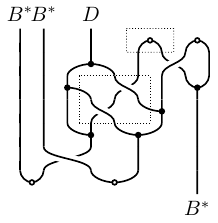}{6863433sp}{7034435sp}{0sp}{7034435sp}}
\expandafter\def\csname ArxivFigureData@0207\endcsname{\ArxivImageBox{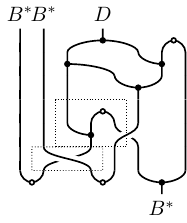}{5884179sp}{7034509sp}{0sp}{7034509sp}}
\expandafter\def\csname ArxivFigureData@0208\endcsname{\ArxivImageBox{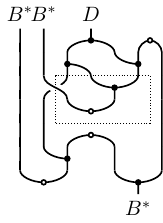}{5137482sp}{7032107sp}{0sp}{7032107sp}}
\expandafter\def\csname ArxivFigureData@0209\endcsname{\ArxivImageBox{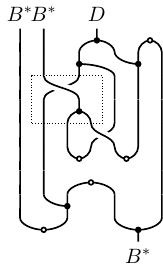}{5138029sp}{8532077sp}{0sp}{8532077sp}}
\expandafter\def\csname ArxivFigureData@0210\endcsname{\ArxivImageBox{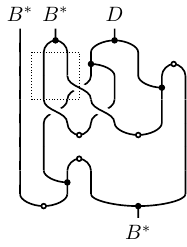}{5884156sp}{7781942sp}{0sp}{7781942sp}}
\expandafter\def\csname ArxivFigureData@0211\endcsname{\ArxivImageBox{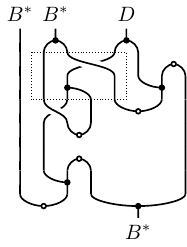}{5884178sp}{7781646sp}{0sp}{7781646sp}}
\expandafter\def\csname ArxivFigureData@0212\endcsname{\ArxivImageBox{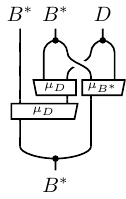}{3968843sp}{6281238sp}{0sp}{6281238sp}}
\expandafter\def\csname ArxivFigureData@0213\endcsname{\ArxivImageBox{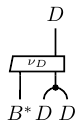}{2624862sp}{4036985sp}{0sp}{4036985sp}}
\expandafter\def\csname ArxivFigureData@0214\endcsname{\ArxivImageBox{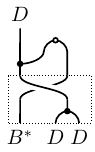}{2997902sp}{4783336sp}{0sp}{4783336sp}}
\expandafter\def\csname ArxivFigureData@0215\endcsname{\ArxivImageBox{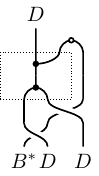}{3122859sp}{5530367sp}{0sp}{5530367sp}}
\expandafter\def\csname ArxivFigureData@0216\endcsname{\ArxivImageBox{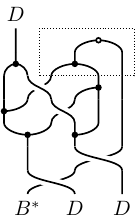}{4354720sp}{7026702sp}{0sp}{7026702sp}}
\expandafter\def\csname ArxivFigureData@0217\endcsname{\ArxivImageBox{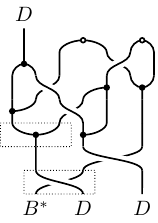}{4987939sp}{7027144sp}{0sp}{7027144sp}}
\expandafter\def\csname ArxivFigureData@0218\endcsname{\ArxivImageBox{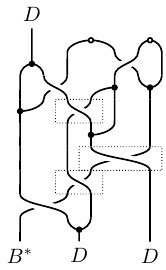}{5236111sp}{8524737sp}{0sp}{8524737sp}}
\expandafter\def\csname ArxivFigureData@0219\endcsname{\ArxivImageBox{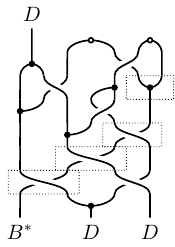}{5496599sp}{7774254sp}{0sp}{7774254sp}}
\expandafter\def\csname ArxivFigureData@0220\endcsname{\ArxivImageBox{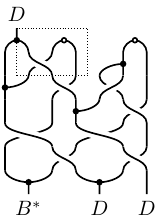}{5132084sp}{7026963sp}{0sp}{7026963sp}}
\expandafter\def\csname ArxivFigureData@0221\endcsname{\ArxivImageBox{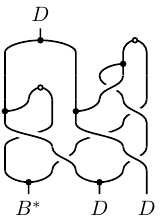}{5132084sp}{7027015sp}{0sp}{7027015sp}}
\expandafter\def\csname ArxivFigureData@0222\endcsname{\ArxivImageBox{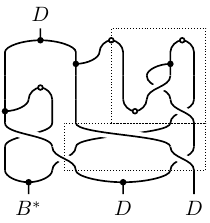}{6624004sp}{7027689sp}{0sp}{7027689sp}}
\expandafter\def\csname ArxivFigureData@0223\endcsname{\ArxivImageBox{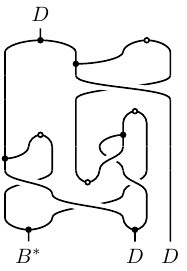}{5878343sp}{8526762sp}{0sp}{8526762sp}}
\expandafter\def\csname ArxivFigureData@0224\endcsname{\ArxivImageBox{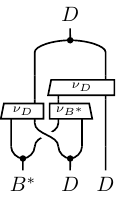}{3833498sp}{6278394sp}{0sp}{6278394sp}}
\expandafter\def\csname ArxivFigureData@0225\endcsname{\ArxivImageBox{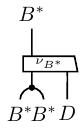}{2624792sp}{4040527sp}{0sp}{4040527sp}}
\expandafter\def\csname ArxivFigureData@0226\endcsname{\ArxivImageBox{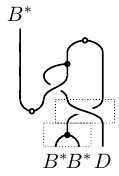}{3743843sp}{5534233sp}{0sp}{5534233sp}}
\expandafter\def\csname ArxivFigureData@0227\endcsname{\ArxivImageBox{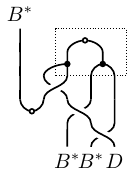}{4117070sp}{5535005sp}{0sp}{5535005sp}}
\expandafter\def\csname ArxivFigureData@0228\endcsname{\ArxivImageBox{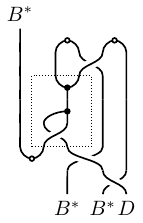}{4490210sp}{7031713sp}{0sp}{7031713sp}}
\expandafter\def\csname ArxivFigureData@0229\endcsname{\ArxivImageBox{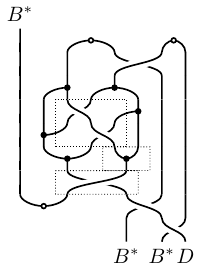}{6355614sp}{8532106sp}{0sp}{8532106sp}}
\expandafter\def\csname ArxivFigureData@0230\endcsname{\ArxivImageBox{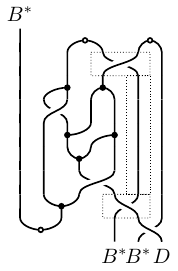}{5609427sp}{8534255sp}{0sp}{8534255sp}}
\expandafter\def\csname ArxivFigureData@0231\endcsname{\ArxivImageBox{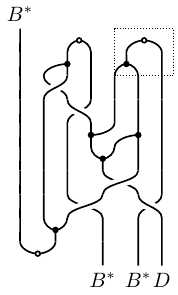}{5610136sp}{9285878sp}{0sp}{9285878sp}}
\expandafter\def\csname ArxivFigureData@0232\endcsname{\ArxivImageBox{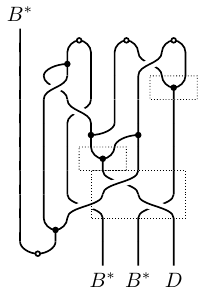}{6243692sp}{9287026sp}{0sp}{9287026sp}}
\expandafter\def\csname ArxivFigureData@0233\endcsname{\ArxivImageBox{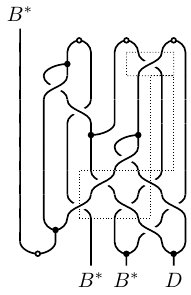}{6013861sp}{9287272sp}{0sp}{9287272sp}}
\expandafter\def\csname ArxivFigureData@0234\endcsname{\ArxivImageBox{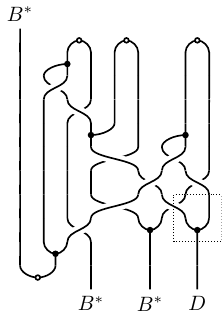}{6989404sp}{10039549sp}{0sp}{10039549sp}}
\expandafter\def\csname ArxivFigureData@0235\endcsname{\ArxivImageBox{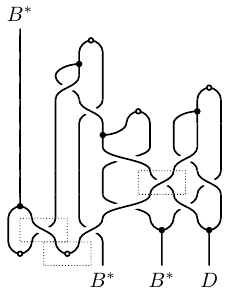}{7134319sp}{9288380sp}{0sp}{9288380sp}}
\expandafter\def\csname ArxivFigureData@0236\endcsname{\ArxivImageBox{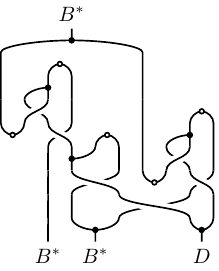}{6898311sp}{8532881sp}{0sp}{8532881sp}}
\expandafter\def\csname ArxivFigureData@0237\endcsname{\ArxivImageBox{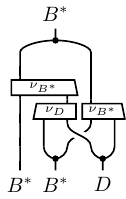}{3968730sp}{6282190sp}{0sp}{6282190sp}}
\expandafter\def\csname ArxivFigureData@0238\endcsname{\ArxivImageBox{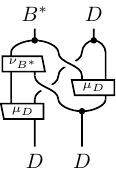}{3636994sp}{5531115sp}{0sp}{5531115sp}}
\expandafter\def\csname ArxivFigureData@0239\endcsname{\ArxivImageBox{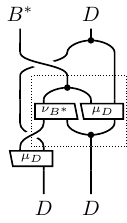}{4003901sp}{7025884sp}{0sp}{7025884sp}}
\expandafter\def\csname ArxivFigureData@0240\endcsname{\ArxivImageBox{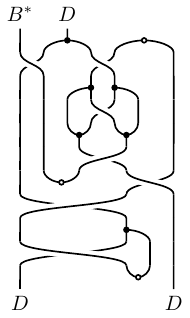}{5982044sp}{10025160sp}{0sp}{10025160sp}}
\expandafter\def\csname ArxivFigureData@0241\endcsname{\ArxivImageBox{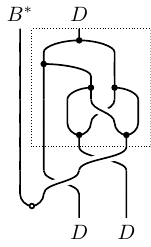}{4750330sp}{7777131sp}{0sp}{7777131sp}}
\expandafter\def\csname ArxivFigureData@0242\endcsname{\ArxivImageBox{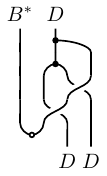}{3371048sp}{5531411sp}{0sp}{5531411sp}}
\expandafter\def\csname ArxivFigureData@0243\endcsname{\ArxivImageBox{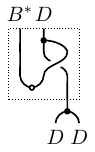}{2997796sp}{4783710sp}{0sp}{4783710sp}}
\expandafter\def\csname ArxivFigureData@0244\endcsname{\ArxivImageBox{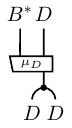}{2251811sp}{4037135sp}{0sp}{4037135sp}}
\expandafter\def\csname ArxivFigureData@0245\endcsname{\ArxivImageBox{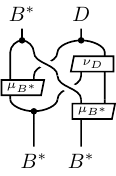}{3658494sp}{5534559sp}{0sp}{5534559sp}}
\expandafter\def\csname ArxivFigureData@0246\endcsname{\ArxivImageBox{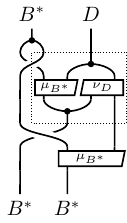}{4003913sp}{7029328sp}{0sp}{7029328sp}}
\expandafter\def\csname ArxivFigureData@0247\endcsname{\ArxivImageBox{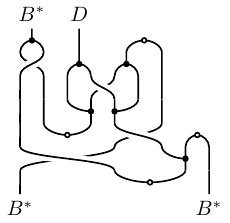}{7236197sp}{7032288sp}{0sp}{7032288sp}}
\expandafter\def\csname ArxivFigureData@0248\endcsname{\ArxivImageBox{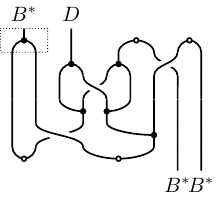}{6989048sp}{6286372sp}{0sp}{6286372sp}}
\expandafter\def\csname ArxivFigureData@0249\endcsname{\ArxivImageBox{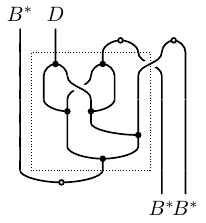}{6490850sp}{7034217sp}{0sp}{7034217sp}}
\expandafter\def\csname ArxivFigureData@0250\endcsname{\ArxivImageBox{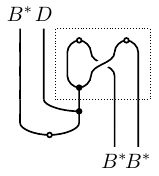}{4998628sp}{5536091sp}{0sp}{5536091sp}}
\expandafter\def\csname ArxivFigureData@0251\endcsname{\ArxivImageBox{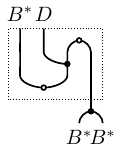}{3879019sp}{4787298sp}{0sp}{4787298sp}}
\expandafter\def\csname ArxivFigureData@0252\endcsname{\ArxivImageBox{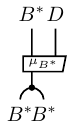}{2251971sp}{4040527sp}{0sp}{4040527sp}}
\expandafter\def\csname ArxivFigureData@0253\endcsname{\ArxivImageBox{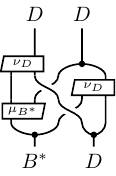}{3636994sp}{5530567sp}{0sp}{5530567sp}}
\expandafter\def\csname ArxivFigureData@0254\endcsname{\ArxivImageBox{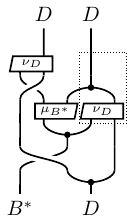}{4004107sp}{7024934sp}{0sp}{7024934sp}}
\expandafter\def\csname ArxivFigureData@0255\endcsname{\ArxivImageBox{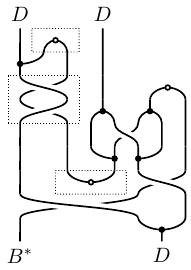}{6013401sp}{8524670sp}{0sp}{8524670sp}}
\expandafter\def\csname ArxivFigureData@0256\endcsname{\ArxivImageBox{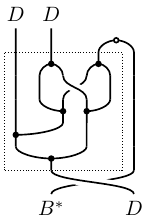}{4727399sp}{7026998sp}{0sp}{7026998sp}}
\expandafter\def\csname ArxivFigureData@0257\endcsname{\ArxivImageBox{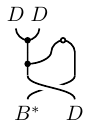}{2862508sp}{4037327sp}{0sp}{4037327sp}}
\expandafter\def\csname ArxivFigureData@0258\endcsname{\ArxivImageBox{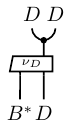}{2251823sp}{4037033sp}{0sp}{4037033sp}}
\expandafter\def\csname ArxivFigureData@0259\endcsname{\ArxivImageBox{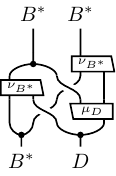}{3636749sp}{5534011sp}{0sp}{5534011sp}}
\expandafter\def\csname ArxivFigureData@0260\endcsname{\ArxivImageBox{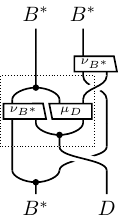}{3868874sp}{7028576sp}{0sp}{7028576sp}}
\expandafter\def\csname ArxivFigureData@0261\endcsname{\ArxivImageBox{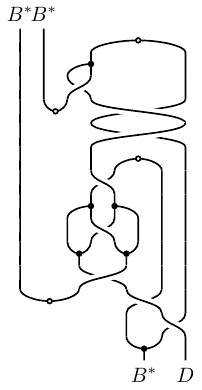}{6354928sp}{12283365sp}{0sp}{12283365sp}}
\expandafter\def\csname ArxivFigureData@0262\endcsname{\ArxivImageBox{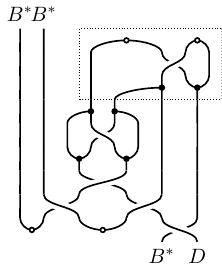}{6989407sp}{8536052sp}{0sp}{8536052sp}}
\expandafter\def\csname ArxivFigureData@0263\endcsname{\ArxivImageBox{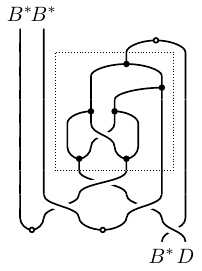}{6355290sp}{8534708sp}{0sp}{8534708sp}}
\expandafter\def\csname ArxivFigureData@0264\endcsname{\ArxivImageBox{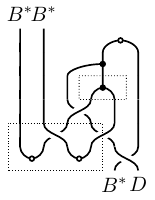}{4862910sp}{6283726sp}{0sp}{6283726sp}}
\expandafter\def\csname ArxivFigureData@0265\endcsname{\ArxivImageBox{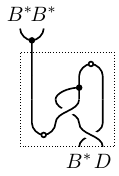}{3743817sp}{5535151sp}{0sp}{5535151sp}}
\expandafter\def\csname ArxivFigureData@0266\endcsname{\ArxivImageBox{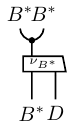}{2251921sp}{4040625sp}{0sp}{4040625sp}}
\expandafter\def\csname ArxivFigureData@0267\endcsname{\ArxivImageBox{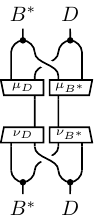}{2939350sp}{7028374sp}{0sp}{7028374sp}}
\expandafter\def\csname ArxivFigureData@0268\endcsname{\ArxivImageBox{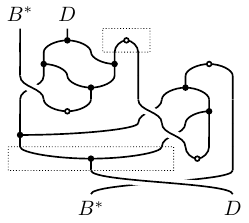}{7846451sp}{7031155sp}{0sp}{7031155sp}}
\expandafter\def\csname ArxivFigureData@0269\endcsname{\ArxivImageBox{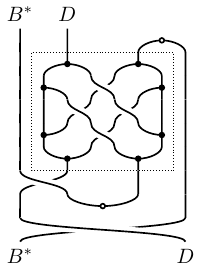}{6354431sp}{8529408sp}{0sp}{8529408sp}}
\expandafter\def\csname ArxivFigureData@0270\endcsname{\ArxivImageBox{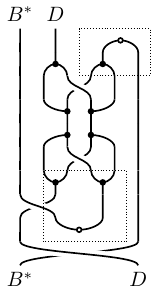}{4862561sp}{9278222sp}{0sp}{9278222sp}}
\expandafter\def\csname ArxivFigureData@0271\endcsname{\ArxivImageBox{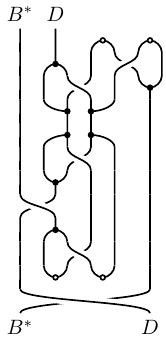}{5235529sp}{10783482sp}{0sp}{10783482sp}}
\expandafter\def\csname ArxivFigureData@0272\endcsname{\ArxivImageBox{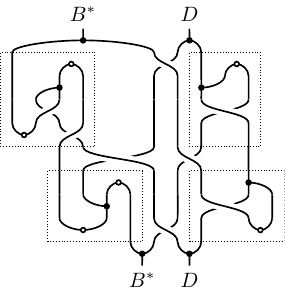}{8980646sp}{9286942sp}{0sp}{9286942sp}}
\expandafter\def\csname ArxivFigureData@0273\endcsname{\ArxivImageBox{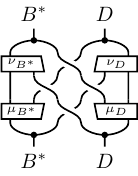}{4356170sp}{5534257sp}{0sp}{5534257sp}}
\expandafter\def\csname ArxivFigureData@0274\endcsname{\ArxivImageBox{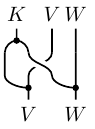}{2968572sp}{4053489sp}{0sp}{4053489sp}}
\expandafter\def\csname ArxivFigureData@0275\endcsname{\ArxivImageBox{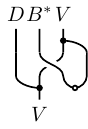}{2768103sp}{4056635sp}{0sp}{4056635sp}}
\expandafter\def\csname ArxivFigureData@0276\endcsname{\ArxivImageBox{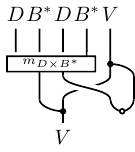}{4255503sp}{4784034sp}{0sp}{4784034sp}}
\expandafter\def\csname ArxivFigureData@0277\endcsname{\ArxivImageBox{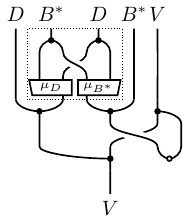}{5747804sp}{7026817sp}{0sp}{7026817sp}}
\expandafter\def\csname ArxivFigureData@0278\endcsname{\ArxivImageBox{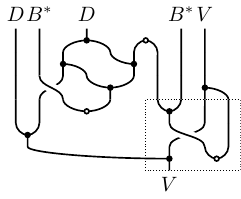}{7599376sp}{6280473sp}{0sp}{6280473sp}}
\expandafter\def\csname ArxivFigureData@0279\endcsname{\ArxivImageBox{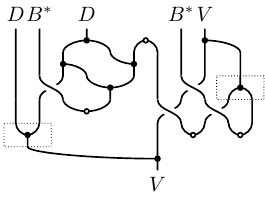}{8346317sp}{6280692sp}{0sp}{6280692sp}}
\expandafter\def\csname ArxivFigureData@0280\endcsname{\ArxivImageBox{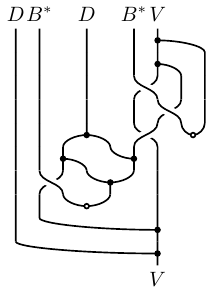}{6495542sp}{9275588sp}{0sp}{9275588sp}}
\expandafter\def\csname ArxivFigureData@0281\endcsname{\ArxivImageBox{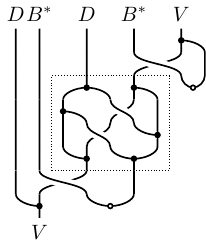}{6494572sp}{7776821sp}{0sp}{7776821sp}}
\expandafter\def\csname ArxivFigureData@0282\endcsname{\ArxivImageBox{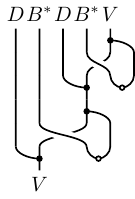}{4255919sp}{6278538sp}{0sp}{6278538sp}}
\expandafter\def\csname ArxivFigureData@0283\endcsname{\ArxivImageBox{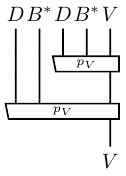}{3966094sp}{5530469sp}{0sp}{5530469sp}}
\expandafter\def\csname ArxivFigureData@0284\endcsname{\ArxivImageBox{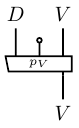}{2473308sp}{4033641sp}{0sp}{4033641sp}}
\expandafter\def\csname ArxivFigureData@0285\endcsname{\ArxivImageBox{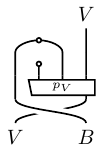}{3203618sp}{4780310sp}{0sp}{4780310sp}}
\expandafter\def\csname ArxivFigureData@0286\endcsname{\ArxivImageBox{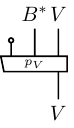}{2325232sp}{4037085sp}{0sp}{4037085sp}}
\expandafter\def\csname ArxivFigureData@0287\endcsname{\ArxivImageBox{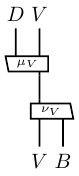}{2474322sp}{5526981sp}{0sp}{5526981sp}}
\expandafter\def\csname ArxivFigureData@0288\endcsname{\ArxivImageBox{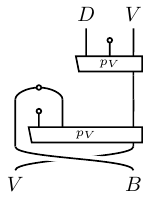}{4695489sp}{6274452sp}{0sp}{6274452sp}}
\expandafter\def\csname ArxivFigureData@0289\endcsname{\ArxivImageBox{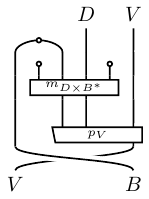}{4695489sp}{6275382sp}{0sp}{6275382sp}}
\expandafter\def\csname ArxivFigureData@0290\endcsname{\ArxivImageBox{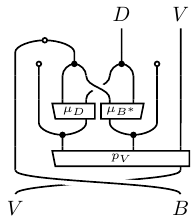}{6187741sp}{7024150sp}{0sp}{7024150sp}}
\expandafter\def\csname ArxivFigureData@0291\endcsname{\ArxivImageBox{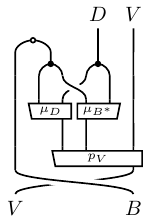}{4695545sp}{7023606sp}{0sp}{7023606sp}}
\expandafter\def\csname ArxivFigureData@0292\endcsname{\ArxivImageBox{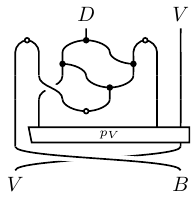}{6187359sp}{6275755sp}{0sp}{6275755sp}}
\expandafter\def\csname ArxivFigureData@0293\endcsname{\ArxivImageBox{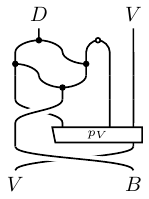}{4695489sp}{6275016sp}{0sp}{6275016sp}}
\expandafter\def\csname ArxivFigureData@0294\endcsname{\ArxivImageBox{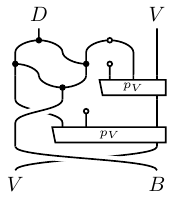}{5441424sp}{6274764sp}{0sp}{6274764sp}}
\expandafter\def\csname ArxivFigureData@0295\endcsname{\ArxivImageBox{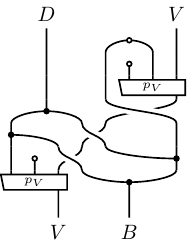}{6055643sp}{7771340sp}{0sp}{7771340sp}}
\expandafter\def\csname ArxivFigureData@0296\endcsname{\ArxivImageBox{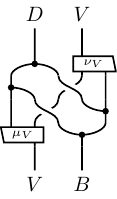}{3685234sp}{6274883sp}{0sp}{6274883sp}}
\expandafter\def\csname ArxivFigureData@0297\endcsname{\ArxivImageBox{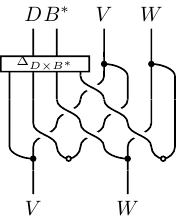}{5552828sp}{7029648sp}{0sp}{7029648sp}}
\expandafter\def\csname ArxivFigureData@0298\endcsname{\ArxivImageBox{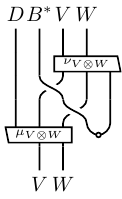}{3833506sp}{6278717sp}{0sp}{6278717sp}}
\expandafter\def\csname ArxivFigureData@0299\endcsname{\ArxivImageBox{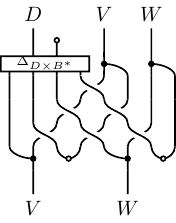}{5552828sp}{7026204sp}{0sp}{7026204sp}}
\expandafter\def\csname ArxivFigureData@0300\endcsname{\ArxivImageBox{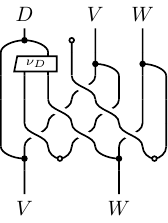}{5276855sp}{7026448sp}{0sp}{7026448sp}}
\expandafter\def\csname ArxivFigureData@0301\endcsname{\ArxivImageBox{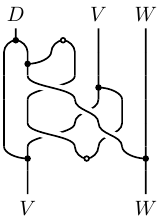}{5176130sp}{7025440sp}{0sp}{7025440sp}}
\expandafter\def\csname ArxivFigureData@0302\endcsname{\ArxivImageBox{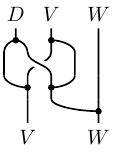}{3683385sp}{4781052sp}{0sp}{4781052sp}}
\expandafter\def\csname ArxivFigureData@0303\endcsname{\ArxivImageBox{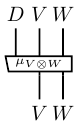}{2564429sp}{4033841sp}{0sp}{4033841sp}}
\expandafter\def\csname ArxivFigureData@0304\endcsname{\ArxivImageBox{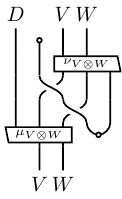}{3833506sp}{6275273sp}{0sp}{6275273sp}}
\expandafter\def\csname ArxivFigureData@0305\endcsname{\ArxivImageBox{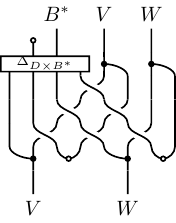}{5552828sp}{7029648sp}{0sp}{7029648sp}}
\expandafter\def\csname ArxivFigureData@0306\endcsname{\ArxivImageBox{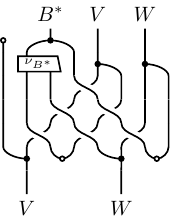}{5351414sp}{7029892sp}{0sp}{7029892sp}}
\expandafter\def\csname ArxivFigureData@0307\endcsname{\ArxivImageBox{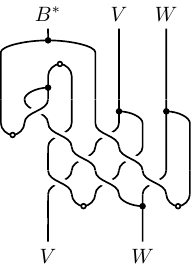}{6022844sp}{8529259sp}{0sp}{8529259sp}}
\expandafter\def\csname ArxivFigureData@0308\endcsname{\ArxivImageBox{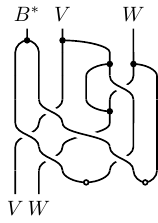}{4986103sp}{7028663sp}{0sp}{7028663sp}}
\expandafter\def\csname ArxivFigureData@0309\endcsname{\ArxivImageBox{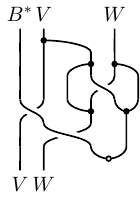}{4391155sp}{6279239sp}{0sp}{6279239sp}}
\expandafter\def\csname ArxivFigureData@0310\endcsname{\ArxivImageBox{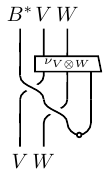}{3222584sp}{5531015sp}{0sp}{5531015sp}}
\expandafter\def\csname ArxivFigureData@0311\endcsname{\ArxivImageBox{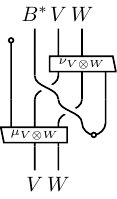}{3685378sp}{6278717sp}{0sp}{6278717sp}}
\expandafter\def\csname ArxivFigureData@0312\endcsname{\ArxivImageBox{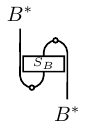}{2760028sp}{4040867sp}{0sp}{4040867sp}}
\expandafter\def\csname ArxivFigureData@0313\endcsname{\ArxivImageBox{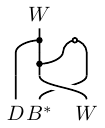}{3310155sp}{4037523sp}{0sp}{4037523sp}}
\expandafter\def\csname ArxivFigureData@0314\endcsname{\ArxivImageBox{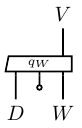}{2564281sp}{4033891sp}{0sp}{4033891sp}}
\expandafter\def\csname ArxivFigureData@0315\endcsname{\ArxivImageBox{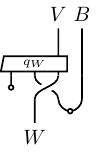}{3072365sp}{4780342sp}{0sp}{4780342sp}}
\expandafter\def\csname ArxivFigureData@0316\endcsname{\ArxivImageBox{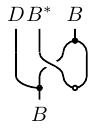}{2857102sp}{4056635sp}{0sp}{4056635sp}}
\expandafter\def\csname ArxivFigureData@0317\endcsname{\ArxivImageBox{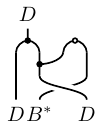}{3245266sp}{4056931sp}{0sp}{4056931sp}}
\expandafter\def\csname ArxivFigureData@0318\endcsname{\ArxivImageBox{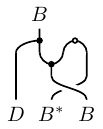}{3229933sp}{4056983sp}{0sp}{4056983sp}}
\expandafter\def\csname ArxivFigureData@0319\endcsname{\ArxivImageBox{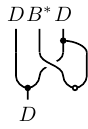}{2768103sp}{4056635sp}{0sp}{4056635sp}}
\expandafter\def\csname ArxivFigureData@0320\endcsname{\ArxivImageBox{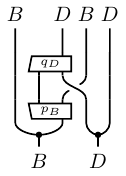}{3966996sp}{5528259sp}{0sp}{5528259sp}}
\expandafter\def\csname ArxivFigureData@0321\endcsname{\ArxivImageBox{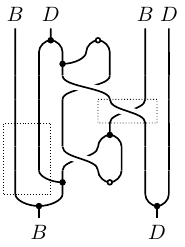}{5832474sp}{7777001sp}{0sp}{7777001sp}}
\expandafter\def\csname ArxivFigureData@0322\endcsname{\ArxivImageBox{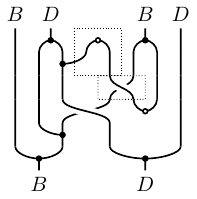}{6204855sp}{6277580sp}{0sp}{6277580sp}}
\expandafter\def\csname ArxivFigureData@0323\endcsname{\ArxivImageBox{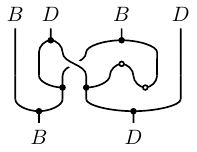}{6204481sp}{4781348sp}{0sp}{4781348sp}}
\expandafter\def\csname ArxivFigureData@0324\endcsname{\ArxivImageBox{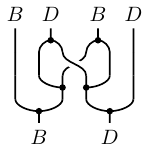}{4712511sp}{4781126sp}{0sp}{4781126sp}}
\expandafter\def\csname ArxivFigureData@0325\endcsname{\ArxivImageBox{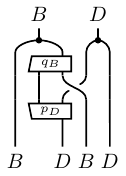}{3966896sp}{5529307sp}{0sp}{5529307sp}}
\expandafter\def\csname ArxivFigureData@0326\endcsname{\ArxivImageBox{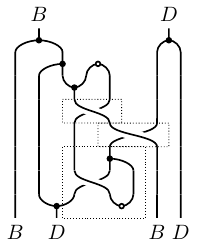}{6205356sp}{7780031sp}{0sp}{7780031sp}}
\expandafter\def\csname ArxivFigureData@0327\endcsname{\ArxivImageBox{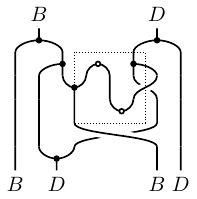}{6205054sp}{6278784sp}{0sp}{6278784sp}}
\expandafter\def\csname ArxivFigureData@0328\endcsname{\ArxivImageBox{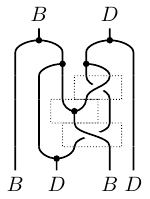}{4712882sp}{6277911sp}{0sp}{6277911sp}}
\expandafter\def\csname ArxivFigureData@0329\endcsname{\ArxivImageBox{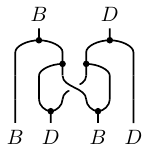}{4712386sp}{4781648sp}{0sp}{4781648sp}}
\expandafter\def\csname ArxivFigureData@0330\endcsname{\ArxivImageBox{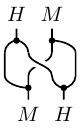}{2536696sp}{4053489sp}{0sp}{4053489sp}}
\expandafter\def\csname ArxivFigureData@0331\endcsname{\ArxivImageBox{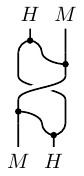}{2645788sp}{5546579sp}{0sp}{5546579sp}}
\expandafter\def\csname ArxivFigureData@0332\endcsname{\ArxivImageBox{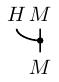}{1845128sp}{2560955sp}{0sp}{2560955sp}}
\expandafter\def\csname ArxivFigureData@0333\endcsname{\ArxivImageBox{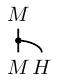}{1845127sp}{2560955sp}{0sp}{2560955sp}}
\expandafter\def\csname ArxivFigureData@0334\endcsname{\ArxivImageBox{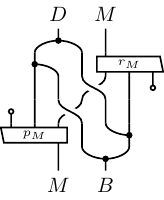}{5177598sp}{6295998sp}{0sp}{6295998sp}}
\expandafter\def\csname ArxivFigureData@0335\endcsname{\ArxivImageBox{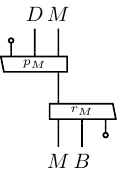}{3685160sp}{5546389sp}{0sp}{5546389sp}}
\expandafter\def\csname ArxivFigureData@0336\endcsname{\ArxivImageBox{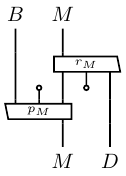}{3976220sp}{5547629sp}{0sp}{5547629sp}}
\expandafter\def\csname ArxivFigureData@0337\endcsname{\ArxivImageBox{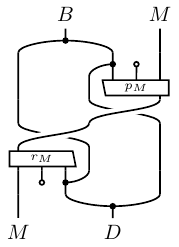}{5629979sp}{7791923sp}{0sp}{7791923sp}}
\expandafter\def\csname ArxivFigureData@0338\endcsname{\ArxivImageBox{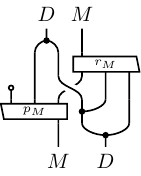}{4431612sp}{5547895sp}{0sp}{5547895sp}}
\expandafter\def\csname ArxivFigureData@0339\endcsname{\ArxivImageBox{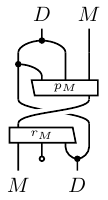}{3391773sp}{6293839sp}{0sp}{6293839sp}}
\expandafter\def\csname ArxivFigureData@0340\endcsname{\ArxivImageBox{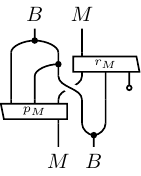}{4431462sp}{5547797sp}{0sp}{5547797sp}}
\expandafter\def\csname ArxivFigureData@0341\endcsname{\ArxivImageBox{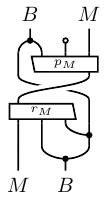}{3391923sp}{6294016sp}{0sp}{6294016sp}}
\expandafter\def\csname ArxivFigureData@0342\endcsname{\ArxivImageBox{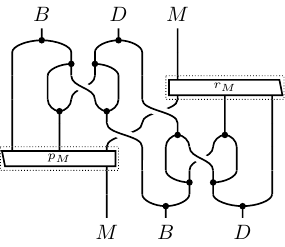}{8979843sp}{7780150sp}{0sp}{7780150sp}}
\expandafter\def\csname ArxivFigureData@0343\endcsname{\ArxivImageBox{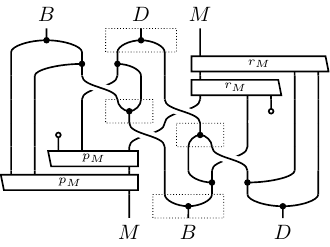}{10400715sp}{7783168sp}{0sp}{7783168sp}}
\expandafter\def\csname ArxivFigureData@0344\endcsname{\ArxivImageBox{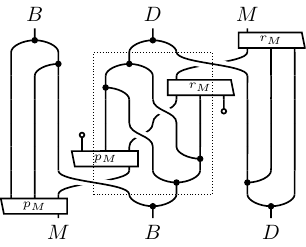}{9653964sp}{7783514sp}{0sp}{7783514sp}}
\expandafter\def\csname ArxivFigureData@0345\endcsname{\ArxivImageBox{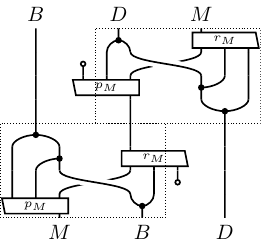}{8233068sp}{7776300sp}{0sp}{7776300sp}}
\expandafter\def\csname ArxivFigureData@0346\endcsname{\ArxivImageBox{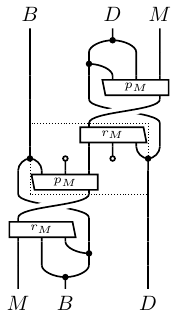}{5620431sp}{10021118sp}{0sp}{10021118sp}}
\expandafter\def\csname ArxivFigureData@0347\endcsname{\ArxivImageBox{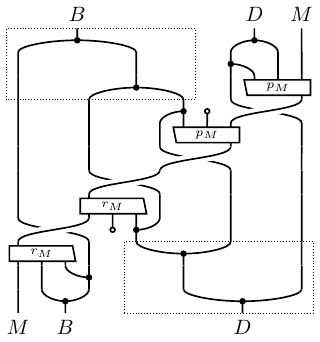}{10096910sp}{10782863sp}{0sp}{10782863sp}}
\expandafter\def\csname ArxivFigureData@0348\endcsname{\ArxivImageBox{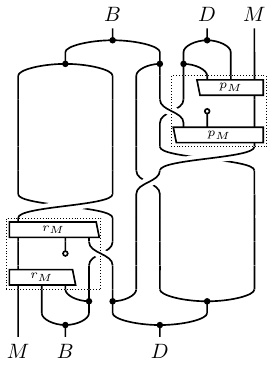}{8604639sp}{11534048sp}{0sp}{11534048sp}}
\expandafter\def\csname ArxivFigureData@0349\endcsname{\ArxivImageBox{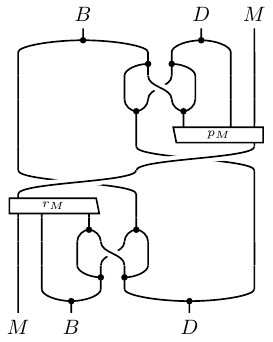}{8604453sp}{10778721sp}{0sp}{10778721sp}}
\expandafter\def\csname ArxivFigureData@0350\endcsname{\ArxivImageBox{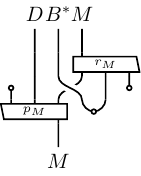}{4431612sp}{5550297sp}{0sp}{5550297sp}}
\expandafter\def\csname ArxivFigureData@0351\endcsname{\ArxivImageBox{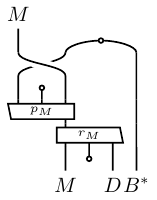}{4946669sp}{6298908sp}{0sp}{6298908sp}}
\expandafter\def\csname ArxivFigureData@0352\endcsname{\ArxivImageBox{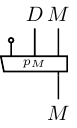}{2414889sp}{4033641sp}{0sp}{4033641sp}}
\expandafter\def\csname ArxivFigureData@0353\endcsname{\ArxivImageBox{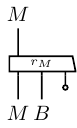}{2414557sp}{4033741sp}{0sp}{4033741sp}}
\expandafter\def\csname ArxivFigureData@0354\endcsname{\ArxivImageBox{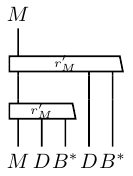}{4190887sp}{5532269sp}{0sp}{5532269sp}}
\expandafter\def\csname ArxivFigureData@0355\endcsname{\ArxivImageBox{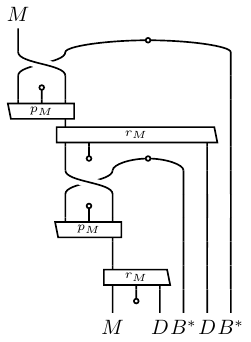}{7923432sp}{10784322sp}{0sp}{10784322sp}}
\expandafter\def\csname ArxivFigureData@0356\endcsname{\ArxivImageBox{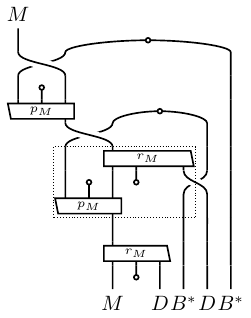}{7923432sp}{10032784sp}{0sp}{10032784sp}}
\expandafter\def\csname ArxivFigureData@0357\endcsname{\ArxivImageBox{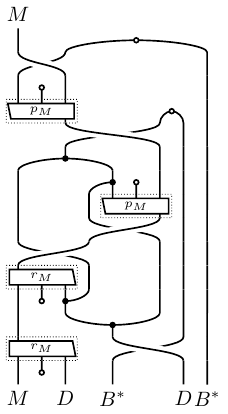}{7176078sp}{13042682sp}{0sp}{13042682sp}}
\expandafter\def\csname ArxivFigureData@0358\endcsname{\ArxivImageBox{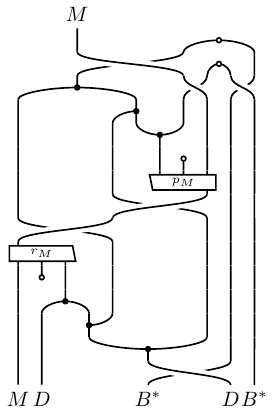}{8666442sp}{13050665sp}{0sp}{13050665sp}}
\expandafter\def\csname ArxivFigureData@0359\endcsname{\ArxivImageBox{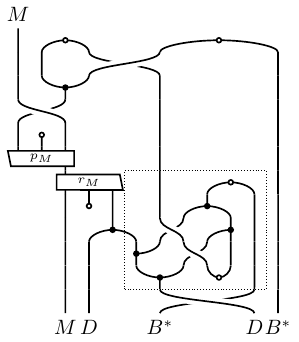}{9413029sp}{10790189sp}{0sp}{10790189sp}}
\expandafter\def\csname ArxivFigureData@0360\endcsname{\ArxivImageBox{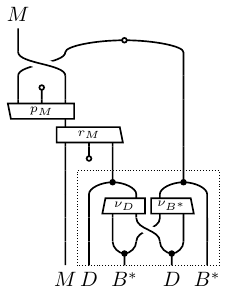}{7175248sp}{9279493sp}{0sp}{9279493sp}}
\expandafter\def\csname ArxivFigureData@0361\endcsname{\ArxivImageBox{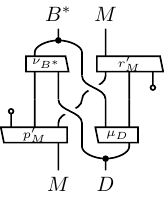}{5177626sp}{6279786sp}{0sp}{6279786sp}}
\expandafter\def\csname ArxivFigureData@0362\endcsname{\ArxivImageBox{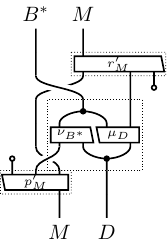}{5248079sp}{7774714sp}{0sp}{7774714sp}}
\expandafter\def\csname ArxivFigureData@0363\endcsname{\ArxivImageBox{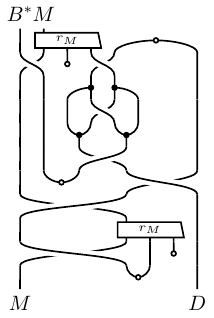}{6728329sp}{10026360sp}{0sp}{10026360sp}}
\expandafter\def\csname ArxivFigureData@0364\endcsname{\ArxivImageBox{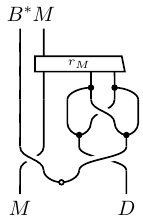}{4489986sp}{7027577sp}{0sp}{7027577sp}}
\expandafter\def\csname ArxivFigureData@0365\endcsname{\ArxivImageBox{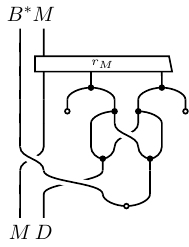}{5958130sp}{7776332sp}{0sp}{7776332sp}}
\expandafter\def\csname ArxivFigureData@0366\endcsname{\ArxivImageBox{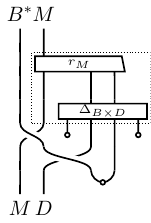}{4750030sp}{7027010sp}{0sp}{7027010sp}}
\expandafter\def\csname ArxivFigureData@0367\endcsname{\ArxivImageBox{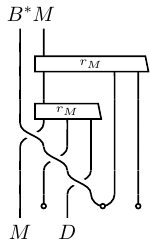}{4714404sp}{7778239sp}{0sp}{7778239sp}}
\expandafter\def\csname ArxivFigureData@0368\endcsname{\ArxivImageBox{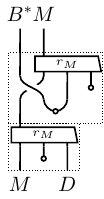}{3257962sp}{6277603sp}{0sp}{6277603sp}}
\expandafter\def\csname ArxivFigureData@0369\endcsname{\ArxivImageBox{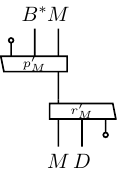}{3685160sp}{5530425sp}{0sp}{5530425sp}}
\expandafter\def\csname ArxivFigureData@0370\endcsname{\ArxivImageBox{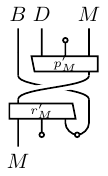}{3382437sp}{5526475sp}{0sp}{5526475sp}}
\expandafter\def\csname ArxivFigureData@0371\endcsname{\ArxivImageBox{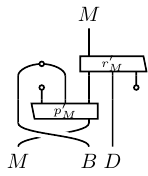}{4652499sp}{5527997sp}{0sp}{5527997sp}}
\expandafter\def\csname ArxivFigureData@0372\endcsname{\ArxivImageBox{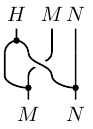}{2909861sp}{4053489sp}{0sp}{4053489sp}}
\expandafter\def\csname ArxivFigureData@0373\endcsname{\ArxivImageBox{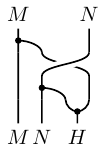}{3337222sp}{4800466sp}{0sp}{4800466sp}}
\expandafter\def\csname ArxivFigureData@0374\endcsname{\ArxivImageBox{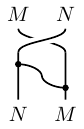}{2645772sp}{4053391sp}{0sp}{4053391sp}}
\expandafter\def\csname ArxivFigureData@0375\endcsname{\ArxivImageBox{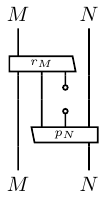}{3327551sp}{6275277sp}{0sp}{6275277sp}}
\expandafter\def\csname ArxivFigureData@0376\endcsname{\ArxivImageBox{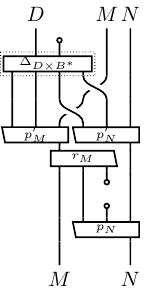}{4634452sp}{9270409sp}{0sp}{9270409sp}}
\expandafter\def\csname ArxivFigureData@0377\endcsname{\ArxivImageBox{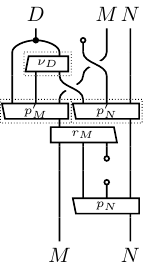}{4634204sp}{8521107sp}{0sp}{8521107sp}}
\expandafter\def\csname ArxivFigureData@0378\endcsname{\ArxivImageBox{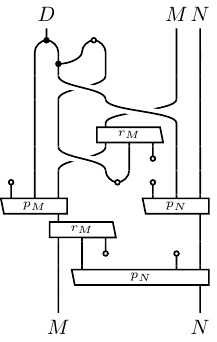}{6837485sp}{10775190sp}{0sp}{10775190sp}}
\expandafter\def\csname ArxivFigureData@0379\endcsname{\ArxivImageBox{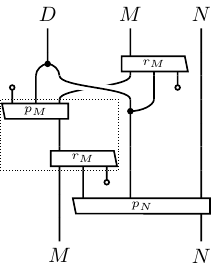}{6872720sp}{8523300sp}{0sp}{8523300sp}}
\expandafter\def\csname ArxivFigureData@0380\endcsname{\ArxivImageBox{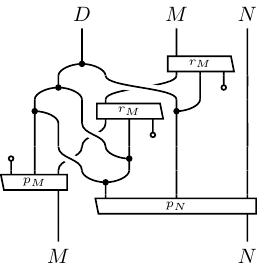}{8329929sp}{8525520sp}{0sp}{8525520sp}}
\expandafter\def\csname ArxivFigureData@0381\endcsname{\ArxivImageBox{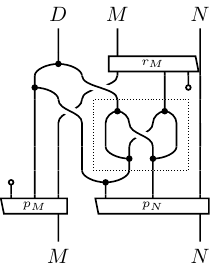}{6837342sp}{8526170sp}{0sp}{8526170sp}}
\expandafter\def\csname ArxivFigureData@0382\endcsname{\ArxivImageBox{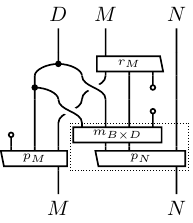}{6091271sp}{7024922sp}{0sp}{7024922sp}}
\expandafter\def\csname ArxivFigureData@0383\endcsname{\ArxivImageBox{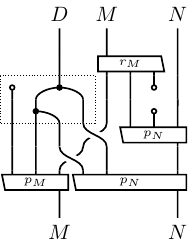}{6126647sp}{7774484sp}{0sp}{7774484sp}}
\expandafter\def\csname ArxivFigureData@0384\endcsname{\ArxivImageBox{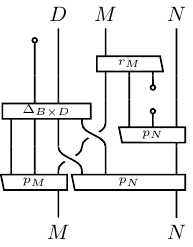}{6091345sp}{7774137sp}{0sp}{7774137sp}}
\expandafter\def\csname ArxivFigureData@0385\endcsname{\ArxivImageBox{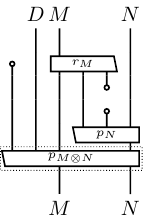}{4634052sp}{7024277sp}{0sp}{7024277sp}}
\expandafter\def\csname ArxivFigureData@0386\endcsname{\ArxivImageBox{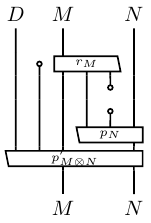}{4746728sp}{7024277sp}{0sp}{7024277sp}}
\expandafter\def\csname ArxivFigureData@0387\endcsname{\ArxivImageBox{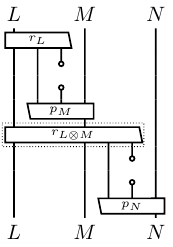}{5435417sp}{7775232sp}{0sp}{7775232sp}}
\expandafter\def\csname ArxivFigureData@0388\endcsname{\ArxivImageBox{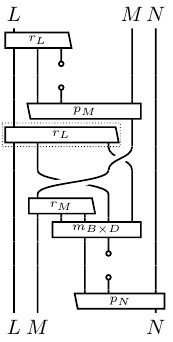}{5435617sp}{10777996sp}{0sp}{10777996sp}}
\expandafter\def\csname ArxivFigureData@0389\endcsname{\ArxivImageBox{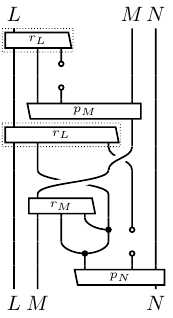}{5435617sp}{10026171sp}{0sp}{10026171sp}}
\expandafter\def\csname ArxivFigureData@0390\endcsname{\ArxivImageBox{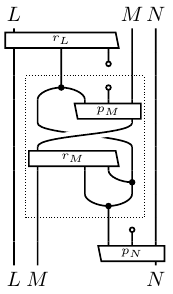}{5435467sp}{9275773sp}{0sp}{9275773sp}}
\expandafter\def\csname ArxivFigureData@0391\endcsname{\ArxivImageBox{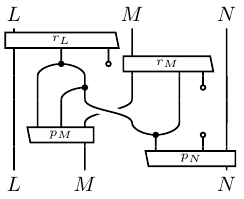}{7673722sp}{6279187sp}{0sp}{6279187sp}}
\expandafter\def\csname ArxivFigureData@0392\endcsname{\ArxivImageBox{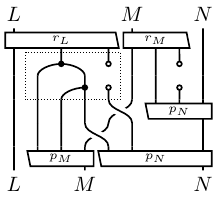}{6927139sp}{6279136sp}{0sp}{6279136sp}}
\expandafter\def\csname ArxivFigureData@0393\endcsname{\ArxivImageBox{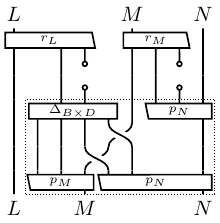}{6927339sp}{7028145sp}{0sp}{7028145sp}}
\expandafter\def\csname ArxivFigureData@0394\endcsname{\ArxivImageBox{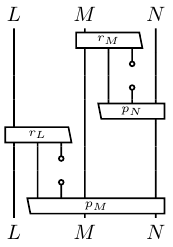}{5435115sp}{7774783sp}{0sp}{7774783sp}}
\expandafter\def\csname ArxivFigureData@0395\endcsname{\ArxivImageBox{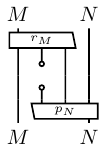}{3327455sp}{4781054sp}{0sp}{4781054sp}}
\expandafter\def\csname ArxivFigureData@0396\endcsname{\ArxivImageBox{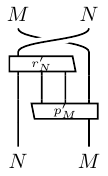}{3382066sp}{5527773sp}{0sp}{5527773sp}}
\expandafter\def\csname ArxivFigureData@0397\endcsname{\ArxivImageBox{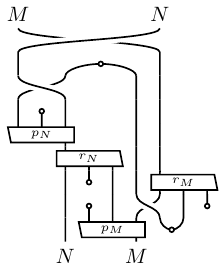}{6892193sp}{8524493sp}{0sp}{8524493sp}}
\expandafter\def\csname ArxivFigureData@0398\endcsname{\ArxivImageBox{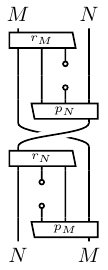}{3382004sp}{8519305sp}{0sp}{8519305sp}}
\expandafter\def\csname ArxivFigureData@0399\endcsname{\ArxivImageBox{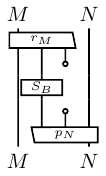}{3327455sp}{5528407sp}{0sp}{5528407sp}}

%% file: main.bbl
\begin{thebibliography}{99}
\setlength{\itemsep}{0em}
\bibitem[BLS15]{BLS15} A. Barvels, S. Lentner, C. Schweigert, \textit{Partially dualized Hopf algebras have equivalent Yetter-Drinfel'd modules}. J. Algebra 430 (2015), 303-342.

\bibitem[Bes95]{Bes95} Y.N. Bespalov, \textit{Crossed modules, quantum braided groups and ribbon structures}. Teoret. Mat. Fiz. 103 (1995) 368-387.

\bibitem[Bes97]{Bes97} Y.N. Bespalov, \textit{Crossed modules and quantum groups in braided categories}. Appl. Categ. Structures 5 (1997), no. 2, 155-204.

\bibitem[BD98]{BD98} Y. Bespalov, B. Drabant, \textit{Hopf (bi-)modules and crossed modules in braided monoidal categories}. J. Pure Appl. Algebra 123 (1998), no. 1-3, 105-129.

\bibitem[BD99]{BD99} Y. Bespalov, B. Drabant, \textit{Cross product bialgebras. I}. J. Algebra 219 (1999), no. 2, 466-505.

\bibitem[BD01]{BD01} Y. Bespalov, B. Drabant, \textit{Cross product bialgebras. II}. J. Algebra 240 (2001), no. 2, 445-504.


\bibitem[BCT13]{BCT13} D. Bulacu, S. Caenepeel, B. Torrecillas, \textit{On cross product Hopf algebras}. J. Algebra 377 (2013), 1-48.

\bibitem[BT14]{BT14} D. Bulacu, B. Torrecillas, \textit{On Doi-Hopf modules and Yetter-Drinfeld modules in symmetric monoidal categories}. Bull. Belg. Math. Soc. Simon Stevin 21 (2014), no. 1, 89-115.

\bibitem[CIMZ00]{CIMZ00} S. Caenepeel, B. Ion, G. Militaru, S. Zhu, \textit{The factorization problem and the smash biproduct of algebras and coalgebras}. Algebr. Represent. Theory 3 (2000), no. 1, 19-42.

\bibitem[CMZ97]{CMZ97} S. Caenepeel,  G. Militaru, S. Zhu, \textit{Crossed modules and Doi-Hopf modules}. Israel J. Math. 100 (1997), 221-247.


\bibitem[Doi92]{Doi92} Y. Doi \textit{Unifying Hopf modules}. J. Algebra 153 (1992), no. 2, 373-385.

\bibitem[EGNO15]{EGNO15} P. Etingof, S. Gelaki, D. Nikshych, V. Ostrik, \textit{Tensor Categories}. Mathematical Surveys and Monographs, 205. American Mathematical Society, Providence, RI, 2015. xvi+343 pp.
\bibitem[HKL26]{HKL26} J.-W. He, X. Kong, K. Li, \textit{The quantum double of Hopf algebras realized via partial dualization and the tensor category of its representations}. J. Pure Appl. Algebra 230 (2026), no. 11, Paper No. 108402.

\bibitem[HS13]{HS13} I. Heckenberger, H.-J. Schneider, \textit{Yetter-Drinfeld modules over bosonizations of dually paired Hopf algebras}. Adv. Math. 244 (2013), 354-394.
\bibitem[HS20]{HS20} I. Heckenberger, H.-J. Schneider, \textit{Hopf algebras and root systems}. Mathematical Surveys and Monographs, 247. American Mathematical Society, Providence, RI, 2020. xix+582 pp.

\bibitem[HS69]{HS69} R.G. Heyneman, M.E. Sweedler, \textit{Affine Hopf algebras. I}. J. Algebra 13 (1969), 192-241.
\bibitem[Kas95]{Kas95} C. Kassel, \textit{Quantum groups}. Graduate Texts in Mathematics, 155. Springer-Verlag, New York, 1995. xii+531 pp.

\bibitem[Li23]{Li23} K. Li, \textit{Partially dualized quasi-Hopf algebras reconstructed from dual tensor categories to finite-dimensional Hopf algebras}. preprint, arXiv:2309.04886.

\bibitem[Maj93]{Maj93} S. Majid, \textit{Braided groups}. J. Pure Appl. Algebra 86 (1993), no. 2, 187-221.

\bibitem[Mol77]{Mol77} R.K. Molnar, \textit{Semi-direct products of Hopf algebras}. J. Algebra 47 (1977), no. 1, 29-51.
\bibitem[Rad85]{Rad85} D.E. Radford, \textit{The structure of Hopf algebras with a projection}. J. Algebra 92 (1985), no. 2, 322-347.

\bibitem[Tak80]{Tak80} M. Takeuchi,
\textit{$\mathrm{Ext}_\mathrm{ad}(\mathrm{Sp}R,\mu^A)\simeq \hat{\mathrm{Br}}(A/k)$}. J. Algebra 67 (1980), no. 2, 436-475.

\bibitem[Tak81]{Tak81} M. Takeuchi, \textit{Matched pairs of groups and bismash products of Hopf algebras}. Comm. Algebra 9 (1981), no. 8, 841-882.

\bibitem[VV94]{VV94} A. Van Daele, S. Van Keer, \textit{The Yang-Baxter and pentagon equation}. Compositio Math. 91 (1994), no. 2, 201-221.

\bibitem[Yet90]{Yet90} D.N. Yetter, \textit{Quantum groups and representations of monoidal categories}. Math. Proc. Cambridge Philos. Soc. 108 (1990), no. 2, 261-290.

\end{thebibliography}
